\documentclass[a4paper,10pt]{amsart}

\usepackage{amssymb}
\usepackage{amsthm}
\usepackage{amsmath}
\usepackage[all]{xy}
\usepackage{color}    

\newcommand{\R}{\ensuremath{\mathbb{R}}}
\newcommand{\Z}{\ensuremath{\mathbb{Z}}}
\newcommand{\A}{\ensuremath{\mathbb{A}}}
\newcommand{\C}{\ensuremath{\mathbb{C}}}

\newcommand{\N}{\ensuremath{\mathbb{N}}}
\newcommand{\K}{\ensuremath{k}}

\newcommand{\Ss}{\ensuremath{\mathcal{S}}}
\newcommand{\Ce}{\ensuremath{\mathcal{C}}}

\newcommand{\B}{\ensuremath{\mathfrak{B}}}

\newcommand{\m}{\ensuremath{\mathfrak{m}}}
\newcommand{\n}{\ensuremath{\mathfrak{n}}}

\newcommand{\Fst}{\ensuremath{\mathcal{F}}}
\newcommand{\Oo}{\ensuremath{\mathcal{O}}}

\newcommand{\Larc}{\ensuremath{\mathcal{L}}}
\newcommand{\Aarc}{\ensuremath{\mathcal{A}}}
\newcommand{\Aarce}{\ensuremath{\mathcal{A}_e}}
\DeclareMathOperator{\w}{ord}
\DeclareMathOperator{\pow}{pow}
\DeclareMathOperator{\W}{\textsc{ord}}
\DeclareMathOperator{\Char}{char}
\newcommand{\p}{\ensuremath{\partial}}
\newcommand{\im}{\ensuremath{\textnormal{im}}}

\newcommand{\id}{\ensuremath{\textnormal{id}}}

\newcommand{\In}{\ensuremath{\textnormal{in}}}

\newcommand{\Sing}{\ensuremath{\textnormal{Sing}}}
\newcommand{\sing}{\ensuremath{\textnormal{sing}}}
\DeclareMathOperator{\Spec}{Spec}

\newcommand{\supp}{\ensuremath{\textnormal{supp}}}

\newcommand{\ra}{\ensuremath{\rightarrow}}
\newcommand{\lra}{\ensuremath{\longrightarrow}}

\newtheorem{Thm}{Theorem}[section]
\newtheorem{Prop}[Thm]{Proposition}
\newtheorem{Lem}[Thm]{Lemma}

\theoremstyle{remark}
\newtheorem*{Rem}{Remark}
\newtheorem*{Rems}{Remarks}
\newtheorem*{Ex}{Example}
\newtheorem*{Exs}{Examples}
\newtheorem*{Alg}{Algorithm}
\newtheorem{Constr}{Construction}
\newtheorem*{Constr*}{Construction}

\title[Arcs, Cords and Felts]{\huge Arcs, Cords and Felts\\ {\large -- Six instances of the Linearization Principle}}

\author{Clemens Bruschek}
\address{University of Vienna, Nordbergstr. 15, 1200 Vienna, Austria}
\email{Clemens.Bruschek@univie.ac.at}
\thanks{Partially supported by project P18992 of the Austrian Science Foundation FWF, by the 3rd tranche of the Doctoral Grant of the University of Innsbruck 2006, the Travelling Grant for Students of the University of Innsbruck 2006/2007 and by the Austrian-Spanish cooperation program Acciones Integradas}

\author{Herwig Hauser}
\address{University of Vienna, Nordbergstr. 15, 1200 Vienna, Austria}
\email{Herwig.Hauser@univie.ac.at}
\thanks{}

\date{\today}

\begin{document}

\maketitle
\begin{abstract}
\noindent It is shown how a selection of prominent results in singularity theory and differential geometry can be deduced from {\it one} theorem, the Rank Theorem for maps between spaces of power series.
\end{abstract}

\tableofcontents



\section{Introduction}
Consider maps $f\colon \C[[x]]^p\to \C[[x]]^q$ between Cartesian products of formal or convergent power series rings in $x=(x_1,\ldots, x_n)$. Such an $f$ is called {\it tactile} if it is given by substitution of power series $a(x)=(a_1(x),\ldots, a_p(x))\in \C[[x]]^p$ for the variables $y=(y_1,\ldots,y_p)$ in a power series vector $g(x,y)\in\C[[x,y]]^q$,

$$f\colon a(x)\mapsto g(x,a(x)).$$

Maps of this type -- and their associated zerosets $\{a\in\C[[x]]^p; g(x,a)= 0\}$ which are called {\it felts} in this paper -- are ubiquitous in local analytic geometry: arc spaces, local automorphism groups of varieties, $K$-equivalence, approximation theorems.\\

The main problem in these settings consists almost always in solving tactile equations  in a space of power series. In the finite dimensional situation, i.e., for ana\-lytic maps between affine spaces $\C^p$ and $\C^q$, the first instance for solving is given by the Implicit Function Theorem, or its more general companion, the Rank Theorem. Both results are used locally at points where the zeroset is smooth. In the infinite dimensional situation which we encounter with spaces of power series, the corresponding theorems are much more subtle. At their best, they allow to solve tactile equations {\it from a certain degree on}, meaning that if one knows an approximate power series solution up to sufficiently high degree, then an exact solution exists and its expansion can be determined.\\

The present paper utilizes and extends in a purely algebraic setting a Rank Theorem for tactile maps as it has been established in \cite{HM}: By local automorphisms of source and target, the tactile map can be linearized in the neighbourhood of a chosen power series (cf. Theorem \ref{thm2} below). The assumptions of the Theorem involve a {\it Rank Condition} and an {\it Order Condition}, which both together are sufficient to yield local linearization. The Rank Condition is also necessary, as it is trivially satisfied by any linear map, and preserved by composition with automorphisms. The Order Condition, formulated in terms of standard bases of ideals, ensures that the higher order terms of $f$ are suitably dominated by its linear terms.  It turns out that the linearizing local automorphisms lie beyond the class of tactile maps. They are of a more general type, so called {\it textile} maps. These can be characterized by saying that the coefficients of the output series are  polynomials in the coefficients of the input series. Actually it turns out that by allowing textile maps as automorphisms, the Rank Theorem can be extended to this more general class of maps between power series spaces (see Theorem \ref{thm2} below).\\

Our emphasis in this paper lies on the {\it universality} of the Rank Theorem. We shall show (see section 5) how six important results in singularity theory and differential geometry are direct corollaries or special cases of this theorem. 
Below, we shall briefly describe these applications. For the convergent case, similar reasoning can be applied, invoking the (analogous) Rank Theorem of Hauser-M\"uller for convergent power series spaces \cite{HM}.\\

Denef and Loeser use the following local triviality result as a main
step to set up motivic integration in the context of singular
varieties \cite{loeser} (see also Lemma 9.1 in
\cite{Looijenga_Motivic_Measures}). Let $X\subseteq \A^N_\K$ be an
affine variety over a field $\K$ of characteristic $0$. Assume $X$ is
given by $f_1,\ldots, f_p \in \K[x_1,\ldots, x_N]$. An arc (or more
precisely a $\K$-arc) of $X$ is a $\K[[t]]$-point of $X$, i.e., a
solution to $f_1=\cdots=f_p=0$ in $\K[[t]]^N$. The set of arcs of $X$
is denoted by $X_\infty$. Similarly, an $m$-jet is a solution of these
equations in $\left(\K[t]/(t)^{m+1}\right)^N$. Write $X_m$ for the set
of $m$-jets of $X$. There is a natural projection $\pi_m\colon X_\infty \ra X_m$. Define $\Aarce=X_\infty \setminus \pi_e^{-1}((\Sing X)_e)$ as the space of arcs of $X$ which do not lie in the singular locus up to order $e$. Moreover assume that $X$ is of pure dimension $d$. Denef and Loeser then show that, for sufficiently large $n \in \N$, the map $$\theta_n\colon\pi_{n+1}(X_\infty)\ra \pi_n(X_\infty)$$ is a piecewise trivial fibration over $\pi_n(\Aarce)$ with fibre $\A^d_\K$, \cite{loeser} (2.5). Specializing to the case of a hypersurface this means the following: Let $\gamma$ be an $n$-jet of $X$ such that not all partial derivatives vanish on $\gamma$ modulo $(t)^{e+1}$. Then finding an $(n+1)$-jet $\eta$ of $X$ with $\eta\equiv \gamma\mod (t)^{n+1}$ for $n\geq e$ sufficiently large can be done by solving a linear system of rank $1$. An extension of this result can be found in \cite{reguera}.\\
The triviality of the above fibration $\theta_n$ is one of the key technical results of the paper \cite{loeser}. Extending the result of Denef and Loeser we show that trivialization can already be obtained on the level of arcs, i.e., the map $$X_\infty \ra \pi_n(X_\infty)$$ is a piecewise affine bundle over $\pi_n(\Aarce)$ (therefore all the $\theta_n$ are simultaneously trivialized). The trivializing map will be textile and can be explicitly described (see section \ref{sectionarcspace}).\\

Grinberg, Kazhdan and Drinfeld show in \cite{Grinberg-Kazhdan} (for $\K=\C$) and \cite{drinfeld} (for arbitrary \K) a similar factorization result: Let $X$ be a scheme over $\K$ and $\gamma_0\in X_\infty\setminus (\Sing X)_\infty$. Then the formal neighbourhood $X_\infty[\gamma_0]$ of $X_\infty$ in $\gamma_0$ is a product of the form $Y[y]\times D^\infty$. Here $Y[y]$ is the formal neighbourhood of a scheme of finite type over $\K$ in a point $y\in Y$ and $D^\infty$ is the product of countably many formal schemes $\hbox{Spf}\ \K[[t]]$. Section \ref{sectiondrinfeld} shows that this local factorization theorem follows from the Rank Theorem (for $\Char \K=0$).\\

Denef and Lipshitz in \cite{denef_lipshitz} and Winkel in \cite{winkel} study power series solutions in one variable $x$ with coefficients in $\K$, $\hbox{char}\ \K = 0$, of a system of algebraic differential equations. Here an algebraic differential equation in variables $x$ and $y_1,\ldots, y_n$ is a differential equation which is polynomial in $x,y_1,\ldots, y_n$ and the derivatives of the $y_i$'s. 
Algebraic differential operators in this sense naturally define a textile map $\K[[x]]^n \ra \K[[x]]^n$. Thus it might prove useful to study the techniques of this paper in the framework of algebraic differential equations. A first instance can be found in section \ref{polynomialvectorfields}. There systems of $n$ explicit differential equations, that is systems of the form $y^{(q)}=P(x,y,y^{(1)},\ldots, y^{(q-1)})$, with a vector of polynomials $P=(P_1,\ldots, P_n)$ and $q\in \N$, are considered. Solving such a system in $\K[[x]]^n$ for given initial conditions, i.e., the coefficients of the solutions are given up to order $q-1$, is equivalent to solving an equation of the form $\ell(y)=b$, where $b \in \K[[x]]^n$ and $\ell$ is a linear textile map.\\

Tougeron's Implicit Function Theorem can be seen as a special case of the Rank Theorem. In fact, the proof of the theorem in \cite{tougeron} is based on the following assertion: Let $F\in \K[[x,y]]$, $\K$ a complete valued field of characteristic $0$, and $F(0)=0$. Denote by $\delta=(\delta_1,\ldots, \delta_p)$ the Jacobian of $F$ w.r.t. $y=(y_1,\ldots, y_p)$ evaluated at $(x,0)$, and for $1\leq i \leq r$ set $y^i=(y_1^i,\ldots, y_p^i)$. Then there are $Y^i\in \K[[x,y_j^l;1\leq l\leq r, 1\leq j \leq p]]^p$, $1\leq i\leq r$, such that $$F(x,\sum_{i=1}^r \delta_i Y^i)=F(x,0)+\delta\left(\sum_{i=1}^r\delta_i y^i \right).$$
In section \ref{sectiontougeron} this will be proven by linearization of the map $$(y^1,\ldots,y^r)\mapsto F(x,\sum_{i=1}^r \delta_i y^i) -F(x,0)$$ using the Rank Theorem.\\

Wavrik's Approximation Theorem, a variation of Artin's Approximation Theorem, is based on Tougeron's Implicit Function Theorem, Thm. I$_n$ in \cite{wavrik_general}, and \cite{artin}. Let $\K$ be a complete valued field of characteristic $0$ and let $F\in \K[[x,y]][z]$ be irreducible, $x=(x_1,\ldots,x_n)$, $y=(y_1,\ldots, y_p)$ and $z=(z_1,\ldots, z_r)$. Wavrik proves that for any integer $q>0$ there exists an $N$ such that if $(\bar y,\bar z)\in \K[[x]]^{p+r}$ satisfies $F(x,\bar y,\bar z)\equiv 0 \mod (x)^N$, then there exist series $(z(x),y(x))\in \K[[x]]^{p+r}$ with 
$$F(x,y(x),z(x))=0$$
and 
$$(y(x),z(x))\equiv (\bar y, \bar z) \mod (x)^q.$$ A solution of $F=0$ can be obtained as follows: Set $y_i(x)=\sum_{\alpha\in\N^n} a^i_\alpha x^\alpha$, $z_i(x)=\sum_{\alpha\in\N^n} b^i_\alpha x^\alpha$ and substitute these expressions in $F(x,y,z)$ for $y$ respectively $z$. This gives for each power of $x$ a polynomial equation in the $a_\alpha^i, b^i_\alpha$. The claim then is that if these equations can be solved up to a sufficiently high degree, then the remaining equations also have a solution. Indeed, it is shown in section \ref{sectiontougeron} that solving the remaining equations is equivalent to solving a system of {\it linear} equations. In one single variable $x$, the equivalence follows directly from our Rank Theorem. In several variables, one can use Tougeron's Implicit Function Theorem.\\ 

Lamel and Mir investigate the following question \cite{lamel_mir}: Let $f\colon(\C^n,0)\ra \C^n$ be a germ of a holomorphic map of generic rank $n$, i.e., its Jacobian determinant is a nonzero power series. For germs $u$ of biholomorphic maps (preserving the origin) we consider the map defined by $u \mapsto f\circ u$. In local coordinates this map is given by substitution of power series, hence a tactile map. Here naturally the question for a left inverse to this map arises. Lamel and Mir prove that it is possible to find one if the derivative $u'(0)$ is known, see \cite{lamel_mir} Theorem 2.4 and section \ref{sectionmirlamel} for the precise statement. This result provides information on the biholomorphic solutions of the equation $f(u)=b(x)$, where $b$ is a holomorphic germ. The formal version of this result is in fact a special instance of the Rank Theorem (for the convergent case, apply the same proof to the result in \cite{HM}).\\ 

It should by now have become clear that the proofs of all the above results are based on one common
principle, {\it linearization}: The respective problems are expressed through certain maps between spaces of power series, and the solution of the problem corresponds to locally linearizing the map at a given point. This process is governed by the Rank Theorem. We now describe its precise statement.\\

Let $\K$ be a field of characteristic $0$. A {\it cord} is a sequence $c=(c_\alpha)_{\alpha\in \N^n}$ of constants $c_\alpha$ in $\K$. The local $\K$-algebra of formal power series $\K[[x]]=\K[[x_1,\ldots, x_n]]$ in $n$ variables and coefficients in $\K$ naturally identifies with the space  $\Ce=\Ce_n(\K)$ of cords over $\K$. The maximal ideal of $\K[[x]]$ of power series without constant term is denoted by $\m$. We define $\w c$ as the order of $c$ as a power series. The space $\Ce$ comes equipped with the $\m$-adic topology induced by the $0$-neighborhoods $\m^k$ of series of order $\geq k$. Open sets in the $\m$-adic topology will be referred to as {\it $\m$-open}, though we will drop the ``$\m$'' if the topology is clear out of context. In the case $n=1$ we shall speak of {\it arcs} and write $\Aarc$ for $\Ce$, respectively $\K[[t]]$ for $\K[[x_1]]$.\\

Elements of $\Ce$ are sequences of elements in $\K$ indexed by $\alpha=(\alpha_1,\ldots, \alpha_n)\in \N^n$. We will consider elements of a Cartesian product $\Ce^p$, $p\in \N$, as sequences in $\K$ which are indexed by $(\alpha,\alpha_{n+1}) \in \N^n\times \N$ with $0\leq \alpha_{n+1}\leq p-1$. A map $f\colon \Ce^m \ra \Ce^p; c\mapsto (f_\alpha (c))_{\alpha\in \N^{n+1}}$ is called {\it textile} if for all $\alpha$ the component $f_\alpha(c)$ is a polynomial in the coefficients of $c$. Let $U$ be a subset of $\Ce^m$. A map $f\colon U\ra \Ce^p$ is called textile if it is the restriction of a textile map $\Ce^m \ra \Ce^p$ to $U$. Denote by $\K[\underline{\N}^m_n]$ the polynomial ring $\K[x_\alpha^j ; 1\leq j\leq m, \alpha\in \N^n]$. The {\it felt defined by a textile map $f\colon \Ce^m\ra\Ce^p$} is the closed subscheme $\Spec \K[\underline{\N}^m_n]/(f_\alpha; \alpha \in \N^{n+1})$ of $\Spec \K[\underline{\N}_n^{m}]$. For most applications in section \ref{sec:Applications} the underlying set of $\K$-points $\{c\in \Ce^m;f(c)=0\}$ is of primary interest. It is an ``algebraic'' subset of a countable Cartesian product of copies of $\K$, where the ``equations'' are infinitely many polynomials in a polynomial ring of countably many variables. Sometimes we will use the term ``felt'' also for Zariski-open subsets of some felt defined by a textile map.\\

Any vector of formal power series $g\in \K[[y_1,\ldots, y_m]]^p$
induces a textile map $f\colon \m\cdot \Ce^m \ra \Ce^p$ via substitution $f(c)=g(c)$. Such maps are called {\it tactile}. Their zerosets in case $n=1$ are called  {\it arc spaces}. The tangent map of $f$ at $0$ is given by the linear terms of $g$.\\ 

Textile maps define by restriction morphisms between felts, thus allowing to speak of the {\it category of felts}. A felt $X$ is {\it smooth at a cord} $c$ if there is an $\m$-adic neighbourhood $U$ of $c$ in $\Ce^m$ such that $X\cap U$ is isomorphic to an \m-open subset of a felt defined by a linear textile map. Tangent spaces and maps can be defined in a natural way.\\


\begin{Exs}

1. As a special case of tactiles, let $X$ be an affine algebraic variety defined in $\A^m_\K$ by polynomials $f_1,\ldots, f_p\in \K[x_1,\ldots,  x_m]$. Then 
$$X_\infty=\{a\in\Aarc^m,\, f_1(a)=\ldots =f_p(a)=0\}$$ 
is the classical arc space associated to $X$ (see \cite{nash}). The point $a(0)$ lies in $X$; we say that the arc {\it goes through} or {\it is centered at} $a(0)$. The image $\pi_q(X)$ of $X$ in $(\Aarc/\m^{q+1})^m$ under the canonical projection is the space of $q$-jets which can be lifted to arcs on $X$, whereas 
$$X_q=\{a\in(\Aarc/\m^{q+1})^m,\, f_1(a)\equiv\ldots \equiv f_p(a)\equiv 0  \textnormal { mod }   \m^{q+1}\}$$
 is the {\it space of $q$-jets on $X$}. \\

2. Let $f\in \K[x_1,\ldots, x_m]$ and identify $\Aarc$ with $\K[[t]]$ via $a=(a_i)_{i\in\N} \mapsto \sum a_i t^i$. Consider the felt $Z$ defined by $f$. For an integer $d>0$, decompose $a\in \Aarc^m$ into $a=\bar a+ \hat a$, where $\bar a$ denotes the expansion of $a$ up to degree $d-1$ (its truncation) and $\hat a$ has order $\geq d$. We identify $\bar a$ as an element in $\pi_{d-1}(Z)$ (see example 1) and consider $f(\bar a+\hat a)=0$ as an equation in $\hat a$, say $g(\hat a)=0$. Of course, $g$ induces a tactile map $G\colon\m^d\cdot \Aarc^m \ra \Aarc$. One of the main aspects of this paper will be to show that for $d$ sufficiently large, the map $G$ can be linearized locally (w.r.t. the $\m$-adic topology) in a well defined way (see Theorem \ref{rankthm}). Intuitively, this will signify that the recursion equations resulting from $f(a)=0$ for the coefficients of $a$ will become eventually as good as linear equations. In particular, if $\bar a$ is an approximate solution for some sufficiently high $d$, then the existence of the ``remainder'' $\hat a$ such that $\bar a + \hat a$ is an exact solution is automatically ensured. Geometrically speaking, this says that the fibre over $\bar a$ under the truncation map $b\mapsto \bar b$ is trivial. For a more general triviality result see Theorem \ref{localtrivialitytheorem}.\\

3. Consider a polynomial recursion of order $d$ indexed by the naturals: %
$$a_i=f_i(a_{i-1},\ldots, a_{i-d}),\hskip 1cm i\geq d, $$
with $f_i\in \K[y_1,\ldots,y_d]$. It induces a textile map 
$$F\colon \Aarc \ra \Aarc\colon a\ra (a_0,\ldots,a_{d-1}, a_d-f_d(a),a_{d+1}-f_{d+1}(a),\ldots ).$$
 Prescribing any initial conditions $a_0,\ldots, a_{d-1}$ gives rise
 to a unique solution $(a_i)_{i\in\N}$ of the recursion. As a
 variation, let the polynomial recursion be of infinite order having
 now the form  $a_i=f_i(a_{i-1},\ldots, a_0)$ with polynomials $f_i\in
 \K[y_1,\ldots,y_i]$. Again we get a textile map $F\colon \Aarc \ra \Aarc$, given by $ a\ra (a_i-f_i(a))_{i\in\N}$.\\


\end{Exs}


The article splits into two parts. In sections 2 to 4 we build up the algebraic apparatus to study felts and textile maps; in section 5 we illustrate the impact of the theory by proving directly the before mentioned results on arc spaces, local analytic geometry and Cauchy-Riemann-manifolds.\\

We start with a textile version of the Inverse Mapping Theo\-rem. In contrast to the corresponding theorem in analysis, the result and its proof are mostly algebraic. The basic assumption is that the non-linear terms of the map increase the order of cords stronger than the linear terms do: Write $f=\ell+h$ with $\ell$ the (invertible) tangent map of $f$ at the cord $c$ in question, and $h$ the higher order terms. The hypothesis is then that the composition $\tilde \ell h$ (with $\tilde \ell$ linear s.t. $\tilde \ell \ell=\id$) is {\it contractive}, i.e., increases the order of cords (see section \ref{sectionformalmaps}):\\

The following theorem is a generalization of the inverse function theorem in \cite{HM} to the setting of textile maps.

\begin{Thm}[Inverse Function Theorem] \label{inverseintro}
Let $U\in \Ce^m$ be an $\m$-adic neighbourhood of $0$ and $f\colon U\subseteq \Ce^m \ra \Ce^m$ be textile with $f(0)=0$. Assume that $f=\ell + h$, $f(U)\subseteq W$ (neighbourhood of 0)	, where $\ell\colon U\ra W$ is linear, textile and invertible, i.e., there exists a linear map $\tilde \ell\colon W\ra U$ with $\ell\circ \tilde \ell=\id_W$ and $\tilde \ell \circ \ell=\id_U$. If $\tilde \ell h $ is contractive on $U$, then $f$ admits on $U$ an inverse.
\end{Thm}

The above result will be further generalized to a parametric version (see \ref{sec:Parametric_Rank_Theorem}) and a version for cords with coefficients in a ring with nilpotent elements (see section \ref{sectionformalmaps}).
	
\begin{Ex}
Consider $f\colon \m\ra \m^2$ given by $a\ra t\cdot a+a^2$. Here we
set $\ell(a)=t\cdot a$ and $\tilde \ell(a)=t^{-1}\cdot a$. For
$a\in\m$, $\tilde \ell h (a) =t^{-1}\cdot a^2$ is not contractive (the
order remains constant for $a\in\m\setminus \m^2$), whereas it is
contractive if we restrict to $a\in\m^2$. It follows from the theorem
that the restriction $f|_{\m^2}\colon \m^2\ra \m^3$ is locally invertible at $0$. 
\end{Ex}

In finite dimensional analysis, the Rank Theorem is a trivial consequence of the Inverse Mapping Theorem. In the infinite dimensional context, say Banach spaces or locally convex spaces, and also in the present situation of  cord spaces, this is no longer true. In many applications it is the Rank Theorem which is actually needed (despite the numerous important applications of theorems like the Nash-Moser Inverse Function Theorem, \cite{Hamilton}) to show the manifold structure of fibres of smooth maps or to prove local triviality.

A first step in this respect is the correct definition of the hypothesis of constant rank. We shall follow in this article the concept proposed by Hauser and M\"uller in [HM] for analytic maps between convergent power series spaces. Working with formal instead of convergent series, the definition slightly simplifies, but is nevertheless somewhat unhandy to grasp. It goes as follows:\\

Let $V$ be an $\m$-adic neighbourhood of $0$ in $\Ce$. A textile map
$\gamma\colon V \subseteq \Ce \ra \Ce^m$ is called a {\it curve} (over
$V$) in $\Ce^m$. If $\im(\gamma) \subseteq U\subseteq \Ce^m$ we say
that $\gamma$ defines a curve {\it in} $U$. Let $f\colon U\subseteq \Ce^m \ra \Ce^p$ be a textile map and let $J$ be a closed linear direct complement in $\Ce^p$ of the image $\im(T_{a_0}f)$ of the tangent map of $f$ at a cord $a_0$. Then $f$ has {\it constant rank at $a_0$ with respect to $J$}  if for all curves $\gamma$ in $U$  with $\gamma(0)=a_0$ and all curves $\eta$ in $\Ce^p$ there exist unique curves $\rho$ in $\Ce^m$ and $\tau$ in $J$ such that 
$$\eta = T_\gamma f \cdot \rho + \tau.$$

If $f$ is tactile, say induced by a polynomial $g\in \K[y_1,\ldots, y_m]$, then $T_\gamma f \cdot \rho$ is given by $$\p_1 g(\gamma)\cdot \rho^1 + \cdots + \p_m g (\gamma)\cdot \rho^m.$$
Here, $\p_i g$ is the $i$-th partial derivative of $g$.\\

In the applications the occurring textile map $f$ is often {\it quasi-submersive}, i.e., if $f=\ell + h$ with $\ell=T_{a_0} f$, then $\im(h)\subseteq \im(\ell)$. If $\ell$ is surjective, $f$ is a {\it submersion}. 
\\

Let $\ell\colon \Ce^m \ra \Ce^p$ be textile linear; a textile linear map $\sigma\colon \Ce^p \ra \Ce^m$ satisfying $\ell \sigma \ell = \ell$ is called a  {\it scission of $\ell$}. Note that a scission provides a left inverse on the image of $\ell$ and the kernel of $\ell\sigma$ defines a closed direct complement of this image. We shall construct scissions for the case of tactile maps using the Grauert-Hironaka-Galligo Division Theorem for formal power series (see section \ref{sectionscission1}). After that we give the construction for arbitrary textile linear maps. This will be achieved by a suitable generalization of the Gauss Algorithm to this context.\\


Here is the extension of the Rank Theorem in \cite{HM} which is needed to prove the results in section \ref{sec:Applications}:

\begin{samepage}
\begin{Thm}[Rank Theorem] \label{thm2}
Let $f\colon U \subseteq \Ce^m \ra \Ce^p$ be a textile map on a neighbourhood $U=\m^l\cdot \Ce^m$ of $0$ with $f(0)=0$. Write $f$ in the form $f=\ell + h$ with $\ell=\textnormal{T}_0f$ the tangent map of $f$ at $0$.
Assume that $\ell$ admits a scission $\sigma\colon \Ce^p\ra \Ce^m$ for which $\sigma h$ is contractive on $U$, and that $f$ has constant rank at $0$ with respect to
$\ker(\ell\sigma)$. Then there exist locally invertible textile maps
$u\colon U\ra  \Ce^m$ and $v\colon f(U) \ra \Ce^p$ at $0$ such that
$$v\circ f\circ u^{-1} = \ell.$$
\end{Thm}
\end{samepage}

This theorem will be extended in sections \ref{sec:Parametric_Rank_Theorem} and \ref{sec:RankTheorem_for_TestRings} to the relative case with parameters and a version for cords with coefficients in a ring with nilpotent elements. In the case of quasi-submersions the order condition gives a sufficient criterion for linearization (see section \ref{sec:Rank_Theorem_Without_Parameters}; and sections \ref{sectiondrinfeld} and \ref{sectiontougeron} for applications). We continue with our list of examples:


\begin{Exs}
5. Consider the tactile map induced by an element $f\in \K[x_1,\ldots, x_m]$.  Let $a$ be an arc of $X$ with $a(0)$ a smooth point of $X=V(f)\subseteq \A_\K^m$. Using the classical Implicit Function Theorem for power series or Theorem \ref{thm2} it's easy to see that $a$ is a smooth cord of $X_\infty$. In section \ref{sectionhypersurfacecase} we treat the more interesting case when $a$ is an arc through a point in the singular locus $\Sing(X)$, but does not lie entirely in it, i.e., $a(0) \in \Sing(X)$ but $a \not \in \Sing(X)_\infty$.\\
6. Consider example 2 above in the special case $f= x + xy \in \K[x,y]$. The induced textile map $F\colon \m\cdot \Aarc^2 \ra \Aarc; a\mapsto f(a)$ has constant rank at $0$ with respect to the complement of $(t)\cdot \Aarc$. The composition $\sigma h$ is contractive, where $h(a^1, a^2)=a^1a^2$ and $\sigma$ is a scission for the linear map $\ell\colon (a^1,a^2)\mapsto a^1$. Thus the Rank Theorem applies and $F$ can be linearized at $0$. The linearizing maps $u$, $v$ of the Rank Theorem are given by $v=\id$ and $u(a)=(a^1(1+a^2),a^2)$. Indeed, this $u$ would be the natural candidate for linearization when viewing $f$ in the form $f=x(1+y)$. Note: It is due to the very simple form of the linear part of $f$ that we can write down $u$ and $v$ so explicitly. Taking $f =tx + xy \in \K[[t]][x,y]$ it is still possible to linearize the induced map on $\m^2$, but it is not possible to give closed formulae for $u,v$. The reason for this lies in the fact that the scission one has to construct is a textile map which is not tactile!\\
7. In several variables it is in general difficult to prove that a textile map has constant rank. The following example gives a simple tactile map which does not have constant rank. By necessity of the rank condition this map is not locally linearizable: Let $F=x^2-y^2 \in \K[x,y]$. Set $\eta=x_1x_2 + \ldots + (x_1x_2)^l$, $l\in \N$, and try to lift the exact solution $(\eta, \eta) \in \K[[x_1,x_2]]^2$ of $F(x,y)=0$. Consider the tactile map
$$g\colon (x_1,x_2)^{l+1}\cdot \K[[x_1,x_2]]^2\ra \K[[x_1,x_2]]; g(x,y)= F(\eta + x, \eta+y).$$
To check whether $g$ has constant rank we have to lift any relation $(a,b)$ between $\p_1 g(0)=\eta$ and $\p_2g(0)=-\eta$ to a relation $(A,B)$ between $\p_1 g(\gamma)$ and $\p_2g(\gamma)$ for any $\gamma=(\gamma_1,\gamma_2)\in \K[[s,x_1,x_2]]$ of sufficiently high order (in the $x_i$) such that $\gamma\equiv 0 \mod (s)$ and $A\equiv a \mod (s)$, $B\equiv b \mod (s)$. Specify $(a,b)=(1,1)$; a short computation shows that then one has to find $A', B'\in (s)\subseteq \K[[s,x_1,x_2]]$ such that

$$(\eta + \gamma_1)A' - (\eta+\gamma_2)B'=\gamma_2-\gamma_1.$$

But this will be impossible (independently of the order of $\gamma$) if $\supp(\gamma_i)\cap \supp(x_1x_2)=\emptyset$. Therefore $g$ does not have constant rank at $0$. Note that the situation does not improve by providing an approximative solution of higher degree, since the above calculation is valid for arbitrary $l\in \N$. However, considering $g$ as a tactile map on $\K[[t]]^2$, it fulfills the rank condition on some $\m$-adic neighbourhood of $0$.
\end{Exs}

\begin{Rem}
  Theorem \ref{inverseintro} and \ref{thm2} may also be used to linearize linear maps, i.e., to reduce a linear textile map to its dominating linear part. This is for example the case when restricting the result of Theorem \ref{Thm:differential_equation} to linear differential equations. The linear textile map $$f\colon \K[[t]]\ra\K[[t]]; a \mapsto a + \frac{d}{dt}(a),$$ induced by the ordinary differential equation $x'+x=0$, is ``linearized'' to $\tilde f(a)=d/dt(a)$.
\end{Rem}

\begin{Rem}
  The above presented technique of linearization may also be applied
  in the situation of cords over a field of positive characteristics
  (or more general over a ring as in section
  \ref{sectiondrinfeld}). One has to treat carefully, though, since for
  example the intuitive characterization of contractiveness as in Lemma
  \ref{Lem:characterize_contractiveness} is then not valid
  anymore. 
\end{Rem}


The authors are indebted to an anonymous referee for pointing out ambiguities in an earlier version of the manuscript, and for encouraging to prove the Grinberg-Kazhdan-Drinfeld formal arc theorem by their methods. They wish to express their gratitude to T. Beck, S. Ishii, B. Lamel, F. Loeser and J. Schicho for stimulating conversations.

\medskip
\bigskip




\section{Felts and textile maps} \label{sectionformalmaps}

In this section we will distinguish subsets of $\Ce^m$, which are given as the vanishing set of {\it textile} maps. For these maps we will provide an inverse mapping theorem and the notion of tangent maps.

\subsection{Textile Maps} \label{firstdefinitions}
Let $\K$ be a field of characteristic $0$. A {\it cord} is a sequence $c=(c_\alpha)_{\alpha\in \N^n}$ of constants $c_\alpha$ in $\K$. Sometimes we will speak more precisely of an {\it $n$-cord over $\K$}. The entry $c_\alpha$ is the {\it coefficient} of $c$ at index $\alpha$ or {\it $\alpha$-coefficient} of $c$. The space  $\Ce=\Ce_n=\Ce(\K)$ of cords over $\K$ naturally identifies with the local $\K$-algebra of formal power series $\K[[x_1,\ldots, x_n]]$ in $n$ variables and coefficients in $\K$. The maximal ideal of $\K[[x]]$ of power series without constant term is denoted by $\m$. Substitution of the variables by power series in $\m$ provides an additional algebraic structure on $\Ce$. We define $\w c$ as the order of $c$ as a power series. The space $\Ce$ comes equipped with the $\m$-adic topology induced by the $0$-neighborhoods $\m^l$ of series of order $\geq l$. Vectors of cords $c=(c^1,\ldots, c^m) \in \Ce^m$ will be encoded in the following way: For the coefficients $c^i_\alpha$, $\alpha \in \N^n$, $i=1,\ldots, m$, we write $c_{(\alpha,i-1)}$ and understand these coefficients as a family indexed by $\N^{n+1}$. Thus a vector of $n$-cords $c \in \Ce^m$ corresponds to an $(n+1)$-cord $(c_\beta)_{\beta \in \N^{n+1}}$ where $c_\beta = 0$ for $\beta_{n+1}\geq m$ and $c_\beta =  c^i_{\bar \beta}$ for $\beta_{n+1}=i-1$, $\bar \beta = (\beta_1,\ldots, \beta_n)$.\\

A {\it textile map} $f\colon \Ce^m \ra \Ce^p$, $m,p\in \N$, maps cords $c=(c_\alpha)_{\alpha \in \N^{n+1}}\in \Ce^m$ to cords with coefficients consisting of polynomials in the $c_\alpha$. If the coefficient polynomials are linear, such mappings are called {linear}. Precisely speaking: $f$ is a textile map if the image of $c\in \Ce^m$ is of the form $$f(c) = (f_\alpha(c))_{\alpha \in \N^{n+1}}$$ with $f_\alpha \in \K[c_\beta, \beta \in \N^{n+1}]$. The $f_\alpha$ is referred to as the $\alpha$-component of $f$. Note: Although the definition of textile maps involves infinitely many variables, in many interesting examples they are given by finite data. See the examples in the introduction. Let $U$ be a subset of $\Ce^m$. A map $f\colon U\ra \Ce^p$ is called textile if it is the restriction of a textile map $\Ce^m \ra \Ce^p$ to $U$. Denote by $\K[\underline{\N}^m_n]$ the polynomial ring $\K[x_\alpha^j ; 1\leq j\leq m, \alpha\in \N^n]$. The {\it felt defined by a textile map $f\colon \Ce^m\ra\Ce^p$} is the closed subscheme $\Spec \K[\underline{\N}^m_n]/(f_\alpha; \alpha \in \N^{n+1})$ of $\Spec \K[\underline{\N}_n^{m}]$. For most applications in section \ref{sec:Applications} the underlying set of $\K$-points $\{c\in \Ce^m;f(c)=0\}$ is of primary interest. It is an ``algebraic'' subset of a countable Cartesian product of copies of $\K$, where the ``equations'' are infinitely many polynomials in a polynomial ring of countably many variables. Sometimes we will use the term ``felt'' also for Zariski-open subsets of some felt defined by a textile map.\\

Let $Y\subseteq \Ce^m$ be a felt. A map $f\colon Y\ra \K$ is called {\it regular at $p\in Y$} if there are a Zariski-open neighbourhood $U$ of $p$ and polynomials $g, h \in \K[\underline{\N}^m_n]$ with $h$ non-zero on $U$ such that

$$f = \frac{g}{h}$$

on $U$. If $f$ is regular at all points of $Y$ we call it {\it regular
  (on $Y$)}. We write $\Oo_Y$ for the $\K$-algebra of {\it regular
  functions on $Y$}. A textile map $f\colon Y \ra \Ce$ is called regular if its components $f_\alpha$ are regular functions on $Y$,i.e., $f_\alpha \in \Oo_Y$ for all $\alpha \in \N^{n}$.\\ 

The sum of textile maps $f,g\colon Y \ra \Ce$ is defined naturally by coefficientwise addition. Moreover we introduce the product of these textile maps by Cauchy-multiplication:

$$f\cdot g := (\sum_{\nu + \mu = \alpha} f_\nu g_\mu)_\alpha.$$

The {\it composition} of textile maps $f\colon \Ce^m \ra \Ce^p$ and $g\colon \Ce^q \ra \Ce^m$ is defined as follows: $f_\alpha \in \K[x_\beta; \beta \in B_\alpha]$, then $$(f\circ g)_\alpha(c):=f_\alpha(g(c)) = f_\alpha(g_\beta(c); \beta \in B_\alpha)$$
for $c\in \Ce$. In each coefficient one has a substitution of polynomials. Call $f$ {\it (right) invertible}, if there exists a textile map $g$ such that $f \circ g=\id$.\\

Any vector of formal power series $g\in \K[[y_1,\ldots, y_m]]^p$
induces a textile map $f\colon \m\cdot \Ce^m \ra \Ce^p$ via substitution $f(c)=g(c)$. Such maps are called {\it tactile}. The felts defined by tactile maps in the case $n=1$ are called  {\it arc spaces}. Note that the composition of tactile maps is the same as the tactile map defined by substitution of the corresponding power series.\\

Sometimes it will be convenient to represent textile linear maps by $\N^2$-matrices, i.e., elements  in $\K^{\N^2}$ : Let $\ell\colon \Aarc \ra \Aarc$ be given by $\ell_i(a)=\sum_{j=0}^\infty \ell_{ij}a_j$. Then we identify $\ell$ with the $\N^2$-matrix $\underline{\ell}:=(\ell_{ij})_{i,j \in \N}$. In analogy to linear algebra the composition of two textile linear maps $\ell^1, \ell^2$ corresponds to the product of the matrices 

$$\underline{\ell}^1\cdot \underline{\ell}^2 = (\sum_{l=0}^\infty \ell^1_{il}\ell^2_{lj})_{i,j}.$$

Note that the last sum is well-defined, since $\ell^j_{il}=0$, $j\in \{1,2\}$, except for a finite number of $l\in \N$.\\

Let $L\colon \R_{>0}^n\ra \R_{>0}$ be a positive linear form. Moreover, let $<$ be a monomial ordering on $\K[[x_1,\ldots, x_n]]$ induced by $L$, e.g., $x^\alpha < x^\beta$ if $L\cdot \alpha < L\cdot \beta$ or $L\cdot \alpha = L\cdot \beta$ and $\alpha < \beta$ lexicographically. Embed $\K[[x_1,\ldots, x_n]]^m$ in $\K[[x_1,\ldots, x_{n+1}]]$ by $x^\alpha\cdot e_i \mapsto x^{(\alpha, i-1)}$. For $\beta \in \N^{n+1}$ define $L\cdot \beta$ as $L\cdot \bar \beta$. Then we get an ordering on $\K[[x_1,\ldots, x_{n+1}]]$ by $x^\alpha < x^\beta$ if: (i) $L\cdot \alpha < L\cdot \beta$ or (ii) $L\cdot \alpha = L\cdot \beta$ and $\alpha < \beta$ lexicographically. This ordering defines a module ordering on $\K[[x_1,\ldots, x_n]]^m$ and thus on $\Ce^m$. Given a module ordering induced by $L$ as above and an element $b \in \K[[x]]^m$ the {\it order of $b$ with respect to $L$} is defined as

$$\w_L b = \min \{L\cdot \alpha; b_\alpha \neq 0\}.$$

Analogously we define the {\it order of a textile map $f\colon \Ce^m \ra \Ce^p$ with respect to $L$} as 

$$\w_L f =\min_{\alpha \in \N^{n+1}} \{L\cdot \alpha; f_\alpha \neq 0\}.$$

Usually, if $L$ has been specified, we simply write $\w b$
respectively $\w f$. If no linear form has been defined, we will
assume $L$ to be $L\colon (a_1,\ldots,a_n)\mapsto 1\cdot a_1 + \ldots
+ 1\cdot a_n$ (as in the beginning of this section). The ideals of
textile maps $f\colon \Ce^m \ra \Ce$ of order greater equal $l$, $l\in \N$, define a natural topology on the ring of textile maps $\Ce^m \ra \Ce$. For obvious reasons we will denote them by $\m^l$. It should be clear from context what is meant by $\m^l$, an ideal of power series or an ideal of textile maps.\\

Consider a textile map $h\colon U\subseteq \Ce^m \lra \Ce^p$, $U$ an $\m$-adic neighbourhood of $0$. Without loss of generality one may assume $h(0)=0$. For $\Gamma \in L(\Z^n)$ the textile map $h$ is called {\it $\Gamma$-shifting on $U$} if

\begin{equation} \label{contraction}
\w\left(h(a) - h(b)\right) \geq \w\left(a-b\right) + \Gamma \ \ \textnormal{for all}\  a,b \in U.
\end{equation}

\begin{Lem}\label{Lem:characterize_contractiveness}
A textile map $h\colon U \subseteq \Ce^m \ra \Ce$ is $\Gamma$-shifting on $U$ if and only if for all $\alpha \in \N^{n}$ the coefficient
$h_\alpha(c)$ is a polynomial in those $c_\nu$ for which $L\cdot \nu \leq
L\cdot \alpha - \Gamma$.
\end{Lem}

For $f\neq0$ on $U$ the {\it degree of contraction of $f$ on $U$ (with respect to $L$)}, denoted by $\kappa(f)$, is defined as the maximal $\Gamma$ such that (\ref{contraction}) is fulfilled. If $f=0$ on $U$, then we set $\kappa(f)=\infty$. We also say $f$ is {\it contractive of degree $\Gamma$ (on $U$ with respect to $L$)}. In case $\Gamma > 0$ the map $h$ is simply called {\it contractive}. By the Lemma above it is easy to see that the contraction degree is well-defined. Let $f, g$ be textile maps of contraction degree $\Gamma$ resp. $\Theta$ on their domains of definition $U$ respectively $V$. Then $\kappa(f+g)\geq \min\{\Gamma, \Theta\}$ on $U\cap V$. Next assume that the composition $f\circ g$ is well-defined and that $\kappa(f|_{g(V)})=\Gamma$. Then $\kappa(f\circ g)=\Gamma + \Theta$ on $V$.\\

It turns out that tactile maps have nice contraction properties. This can be seen in the following example:

\begin{Ex}
Consider the ``monomial'' map $f\colon U\subseteq \Ce^m \ra \Ce; c \mapsto
c^\alpha$, $\alpha \in \N^{m}$, with $U=\m^{p_1}\times \cdots
\times \m^{p_m}$. Denote by $p_{\min}$ the minimum of the $p_i$'s. It is immediate that $\kappa(f)=\sum_{i=1}^m \alpha_i p_i - p_{\min}$.
\end{Ex}
\medskip

This example shows that one has some control over the degree of contraction of tactiles. By shrinking the domain of definition $U$ the contraction degree increases!\\

The notion of contractivity will be used in the following more general context (see section \ref{sectiondrinfeld}). Let $A$ be a local, commutative and unital $\K$-algebra with nilpotent maximal ideal $\n$, i.e., there exists an integer $N\in \N$ such that $\n^N=0$. The set of cords with coefficients in $A$ will be denoted by $\Ce_A$. For $a\in A$ set 
$$\pow(a)=\left\{
\begin{array}{ll}
\max\{i; a\in \n^i\} & a\neq 0\\
N & a=0
\end{array}
\right. 
$$
where $\n^0=A$. For an element $c\in \Ce^m_A$ we define the {\it refined order of $a$} as the pair
$$\W(c)=(\w(c),\pow(c_{\w c})) \in \N^2, $$

where $\N^2$ is considered with the lexicographic ordering. If $(c_i)_{i\in \N}$ is a family of cords in $\Ce^m_A$ with increasing sequence of modified orders $(\W(c_i))_i$, then the family is summable. A map $h\colon U\subseteq \Ce_A^m\ra\Ce^p_A$, $U$ an $\m$-adic neighbourhood of $0$, is called contractive if for all cords $a$ and $b$ in $U$
$$\W(h(a)-h(b))>\W(a-b)$$
holds. If $A=\K$ the only nilpotent ideal in $A$ is the zero-ideal and both notions of contractivity coincide. 
\begin{Rem}
 For other applications it might be reasonable to interchange the order of $\w$ and $\pow$ in the definition of the modified order.
\end{Rem}

\begin{samepage}
\begin{Thm}[Inverse Mapping Theorem] \label{inversemappingthm}
Let $U\in \Ce_A^m$ be an $\m$-adic neighbourhood of $0$ and $f\colon
U\subseteq \Ce_A^m \ra \Ce_A^m$ be textile with $f(0)=0$. Assume that
$f=\ell + h$, $f(U)\subseteq W$ (neighbourhood of 0), where
$\ell\colon U\ra W$ is textile, linear and invertible, i.e., there
exists a linear map $\tilde \ell\colon W\ra U$ with $\ell\circ \tilde
\ell=\id_W$ and $\tilde \ell \circ \ell=\id_U$. If $\tilde \ell h $ is
contractive on $U$, then $f$ admits on $U$ an inverse, i.e., there is
a textile map $g\colon f(U) \ra U$ with $f\circ g=\id_{f(U)}$ and $g\circ f=\id_U$. 
\end{Thm}
\end{samepage}

\begin{proof}
First simplify to the case $f=\id + h$ with $h$ contractive on $U$ by considering $\tilde \ell f = \id + \tilde \ell h$. Then recursively define a sequence $\left(g_j\right)_{j\in \N}$ of textile maps $U \ra \Ce_A^m$ by $g_0=0$ and $g_{j+1} = \id - h\circ g_j$ for $j\geq1$. The sequence $(g_j)_j$ converges pointwise on $U$. Indeed, $g_{j+1}$ can be written as
$$g_{j+1}=g_0+\sum_{i=1}^{j+1} D_i,$$
with $D_j=g_{j} - g_{j-1}$. For any $a\in U$ the family $\left(D_j(a)\right)_{j\in \N}$ is summable: For $j\geq 2$ we can write $D_j$ as
$$D_j=h\circ g_{j-2} - h\circ g_{j-1},$$ and thus by contractivity of $h$ 
$$\W D_j(a)>\W D_{j-1}(a)$$
for any $a\in U$. Therefore, $g=\lim g_j$ is a well-defined textile map $U\ra \Ce_A^m$. It remains to show that $g$ is a (right-) inverse for $f$. But for any $j\in \N$ 
$$f\circ g - \id=(g-g_j) + (h\circ g - h\circ g_j) + (h\circ g_j - h\circ g_{j-1}),$$
and each of the three summands on the right hand side tends to $0$ for $j\ra \infty$. Therefore, $g$ is a right-inverse for $f$.\\ It's not hard to see that $g$ is in fact also a left-inverse for $f$. It suffices to show that $$\left(\W(g_m(f(a))-a)\right)_{m\in \N}$$ is an increasing sequence (for all $a\in U$) which follows immediately from contractiveness of $h$.
\end{proof}

\begin{Rems}
(a) Note that there are textile invertible maps which are not of
the special form demanded in the theorem. An example would
be $$f\colon \Aarc \ra \Aarc; (a_i)_i \mapsto (a_0 + a_1^2, a_1, a_2 +
a_3^2, a_3, + \ldots ).$$
Nevertheless this textile map may be transformed into a form which allows the use of Theorem \ref{inversemappingthm}.\\
(b) Theorem \ref{inversemappingthm} implies as a special case the classical Inverse
Mapping Theorem for formal power series: Let $f=(f_1,\ldots,f_n) \in \K[[x_1,\ldots,x_n]]^n$ with $f(0)=0$. Denote by $\p f$ its Jacobian matrix $\left( \p f_i/\p x_j\right)_{i,j}$. If $\det \p f(0) \neq 0$ then there exist $g_1,\ldots, g_n \in \m\cdot \K[[x_1,\ldots, x_n]]$ with $$f(g_1(x),\ldots, g_n(x))=(x_1,\ldots, x_n).$$
\end{Rems}

\begin{Lem}\label{lem:contractive_map_nilpotent_sufficient_criterion}
 Assume that $h\colon U\subseteq \Ce^m_A \ra \Ce_A$, $U$ an $\m$-adic neighbourhood of $0$, is textile with $h(0)=0$. Moreover, let its $\alpha$-coefficient $h_\alpha$ be of the form $h_\alpha=h'_\alpha + h''_\alpha$ with $h'_\alpha \in A[x_l^j;1\leq j \leq m, 0\leq l< e]$ and $h''_\alpha \in \n[x^j_e,1\leq j \leq m,]$, where $e=|\alpha|$. Then $h$ is contractive. 
\end{Lem}

\begin{proof}
 For simplicity of notation we restrict to the case $h\colon \Aarc_A^m\ra \Aarc_A$. Let $a,b \in U$ with $\w(a-b)=i$. Write $a=b+\zeta$ for some $\zeta \in U$, $\zeta_i\neq 0$. By assumption on $h$ we have $h_j(a)-h_j(b)=0$ for $0\leq j\leq i-1$. Using Taylor expansion we see for the $i$-th coefficient
\begin{equation}\label{eq:taylor_h_zweistrich}
h_i(b+\zeta)-h_i(b)= \p h''_i(b)\cdot \zeta_i + \p^2h''_i(b)\cdot \zeta_i^2 + \ldots
\end{equation}

where $\p^lh''_i(b)\cdot \zeta_i^l$ abbreviates all order $l$ terms in the Taylor expansion. By assumption on $h$ all partial derivatives in (\ref{eq:taylor_h_zweistrich}) lie in $\n$, thus $$\pow(h_i(a)-h_i(b))\geq \pow(\zeta_i)+1>\pow((a-b)_i),$$
so $h$ is contractive.
\end{proof}

Denote by $\Ce(q)$ the {\it set of $q$-cords}, that is $$\Ce(q)=\{(c_\alpha)_{\alpha\in \N^n} \in \K^{n(q+1)}; |\alpha|<q+1\}\cong \Ce\mod \m^{q+1}.$$ 
Analogously we introduce $\Ce^m(q)$, and especially if $X$ is the felt defined by a textile map $f\colon \Ce^m \ra \Ce$ the {\it set of $q$-cords of $X$} $$X(q)=\{(c_\alpha)_{\alpha\in \N^{n+1}} \in (\K^{n(q+1)})^m; |\alpha|<q+1, f(c)=0 \mod \m^{q+1}\}.$$ Clearly, $X(q)$ has a natural structure as a subscheme of $\A^{n(q+1)m}$. The canonical projections $X\ra X(q)$, resp. $X(q)\ra X(p)$ for $q\geq p$, induced by truncation of cords, resp. $q$-cords,  will be denoted by $\pi_q$, resp. $\pi^q_p$. We will identify $\Ce(q)$ with those cords $c\in \Ce$ such that $c_\alpha=0$ for all $\alpha\in \N^n$ with $|\alpha|>q$. Then a textile map $f\colon \Ce^m \ra \Ce^r$ induces a map $f_q\colon \Ce(q)^m\ra\Ce(q)^r$ with components $(f_q)_\alpha$. It's easy to see that if $f$ is $0$-shifting, then $f$ is compatible with the truncation maps $\pi_q$ in the sense that the following diagram is commutative:

$$\xymatrix@1{ \Ce^m \ \ar@{->}[r]^{f} \ar@{->}[d]_{\pi_q} & \ \Ce^r \ar@{->}[d]^{\pi_q}\\
\Ce^m(q)\ar@{->}[r]^{f_q} \ar@{->}[d]_{\pi^q_p} &\ \Ce^r(q) \ar@{->}[d]_{\pi^q_p}\\
\Ce^m(p) \ar@{->}[r]^{f_p} & \ \Ce^r(p) }$$ 

Let $f\colon \Ce^m \ra \Ce^p$ be a textile map. The felt $X=f^{-1}(0)$ is called {\it smooth at a cord} $c \in X$ if there is an \m-open neighbourhood $U$ of $c$ such that $X\cap U$ is isomorphic to the kernel of a linear textile map. Smoothness of a ($\K$-)point in classical algebraic geometry can be embedded into the theory of felts in the following way: Let $X$ be a subvariety of $\A^n_\K$ given by the ideal $\langle f_1, \ldots, f_p\rangle \subseteq \K[x_1,\ldots, x_n]$. Consider the tactile map

$$F_X\colon \Aarc^n \ra \Aarc^p; a \mapsto \left(f_1(a), \ldots, f_p(a)\right).$$ The $\K$-points of $X$, i.e., elements $b=(b^1,\ldots, b^n)\in \K^n$ with $f_i(b)=0$ for all $i$, correspond to arcs $a=(a^1,\ldots, a^n)\in X_\infty$ with $a^i=(b^i,0,\ldots)$.
Using the classical Implicit Function Theorem or the Rank Theorem from section \ref{sectionrankthm} it is obvious that $X_\infty$ is smooth at $a$ if $b$ is a smooth point of $X$. With the same methods one shows that all arcs $a\in X_\infty$ with $a(0)$ a smooth point of $X$ are smooth cords of $X_\infty$.


\subsection{Differential calculus of textile maps} \label{differentialcalc}
This section provides the basic definitions for the differential
apparatus in the context of textile maps. Consider a textile map
$f\colon \Ce^m \ra \Ce; c\mapsto \left(f_\alpha(c)\right)_\alpha$. For $\nu \in \N^{n+1}$ the {\it partial derivative of $f$ with respect to $\nu$} is defined by coefficientwise differentiation: 
$$\p_\nu f(c) := \left(\frac{\p f_\alpha(c)}{\p c_\nu}\right)_{\alpha \in \N^n}.$$ 

Obviously $\p_\nu$ is a $\K$-derivation on $\Ce$. We write $\p f(a)$ for the vector $(\p_\nu f(a))_{\nu \in \N^{n+1}}$. The {\it directional derivative of $f$ in $a\in \Ce^m$ in direction of $v \in \Ce^m$}, is given by $$T_a f\cdot v := \left.\frac{1}{s}\left( f(a+sv) - f(a)\right) \right|_{s=0}.$$ Using the notation $(\p f)(a)\bullet v := \sum_{\nu\in \N^{n+1}} \p_\nu f(a) v_\nu$ this yields for a textile map $f$: $$T_a f\cdot v = (\p f)(a)\bullet v.$$

In the case of tactile maps the directional derivative can be calculated as usually. Let $F: \Ce^m \ra \Ce$ be given by $c\mapsto f(c)$, with $f\in \K[x_1,\ldots, x_m]$, then $$T_a F\cdot v =\p_1f(a)\cdot v_1 + \cdots + \p_m f(a)\cdot v_m.$$

It's easy to see that each $f$ has a Taylor expansion around $a$ given by $$f(y) = \sum_{\nu \in \N^{n+1}}\frac{1}{\nu!}(\p_\nu f)(a)(y-a)^\nu.$$

Consider textile maps $f:\Ce^m\ra \Ce^p$ and $g:\Ce^l \ra \Ce^m$ with well-defined composition $f\circ g$. The chain-rule for partial derivatives is given by:  $\p_\nu \left(f\circ g\right)(a) = \p f(g(a))\bullet \p_\nu g(a)$. 
More generally one deduces the following properties of the derivative:

\begin{Prop}
With the above assumptions the following holds:

1. $\p(f\circ g)(a) = (\p f(g(a))\bullet \p_\nu g(a))_\nu$\\
2. Let $\phi$ be a textile linear map, then $$T_a(\phi\circ f) \cdot v = \phi(T_a f\cdot v).$$
3. Let $\phi$ be of the form $\phi = \id + g$, and $f\circ g$ well-defined, then

\begin{eqnarray*}
T_a(f\circ \phi) \cdot v & = & \p f (\phi(a))\bullet v + \p
f(\phi(a)) \bullet
\left(\p g (a) \bullet v\right)\\
& = & T_{\phi(a)} f \cdot v + T_{\phi(a)}f\cdot \left(T_a
g(a)\cdot v\right)\ .
\end{eqnarray*}
\end{Prop}

\begin{proof}
For the easy calculations see \cite{bruschek}.
\end{proof}

\medskip

\bigskip


\section{Scissions of linear maps} \label{sectionscission}

We recall the notion of scissions of linear maps (cf. \cite{HM}, p. 99). In Theorem \ref{inversemappingthm} the inverse of a map is constructed using the inverse of its tangent map. For the proof of the Rank Theorem (see section \ref{sectionrankthm}) a similar construction is needed. There the linearizing automorphisms rely on inverting the tangent map on its image (and going back to a direct complement of its kernel).\\
Let $\ell:V \ra W$ be a linear map of $\K$-vector spaces $V$ and $W$. A {\it scission} of $\ell$ is a linear map $\sigma:W \ra V$ with $\ell \sigma \ell = \ell$. Scissions provide projections $\pi_{\ker}:= \id -\sigma \ell$ onto $\ker(\ell)$ and $\pi_\im := \ell\sigma$ onto $\im(\ell)$. Thus, $\sigma$ induces direct sum decompositions $$V=\ker(\ell)\oplus \im (\sigma \ell) \ \ \textnormal{and} \ \ W=\im(\ell)\oplus \ker(\ell \sigma).$$ On the other hand each direct sum decomposition of the form $$V=\ker(\ell)\oplus L, \ \ W=\im(\ell)\oplus J$$ induces a scission $\sigma$ of $\ell$ by $$\sigma:=\left( \ell|_L\right)^{-1}\circ \pi,$$ where $\pi$ is the projection from $W$ onto $\im(\ell)$ with kernel $J$. In the following sections we will explicitly construct scissions for textile linear maps. This is necessary to obtain precise information on the degree of contraction of the scission used to linearize a map of constant rank. For the case of $\K[[x]]$-linear maps $\ell$, the degree of contraction will be given by the maximal order of a standard basis of the module $\im(\ell)$.\\

\subsection{Scissions of $\K[[x]]$-linear maps} \label{sectionscission1}
We start with a short reminder on standard bases. Most of the notation has already been introduced in section \ref{firstdefinitions}. Write $x^{\alpha,j}$, $\alpha \in \N^n$, $j\in \{1,\ldots,m\}$, for the vector $(0,\ldots, x^\alpha, 0,\ldots, 0)^\top$ in $\K[[x]]^m$, where $x^\alpha$ is at the $j$-th component. For $\alpha \in \N^n$ we use the usual multiindex notation, i.e., $x^\alpha=x_1^{\alpha_1}\cdots x_n^{\alpha_n}$. Thus, each element $a\in \K[[x]]^m$ can be uniquely written as $$a=\sum_{\alpha, j} a_{\alpha,j} x^{\alpha,j}.$$ Fix a monomial order $<$ on $\N^n$. We assume that it is induced by a positive linear form $L$ on $\R_{>0}^n$. Extend this order to a module ordering on $\K[[x]]^m$, again denoted by ``$<$''. The {\it support} of an element $a\in \K[[x]]^m$ is defined as $$\textnormal{supp}(a):=\{(\alpha, j) \in \N^{n+1}; a_{\alpha,j}\neq 0\}.$$ By $\In(a)$ we denote the initial monomial vector $x^{\alpha,j}$ of $a$, i.e., $(\alpha,j)=\min_< \textnormal{supp}(a)$. An element $a\in \K[[x]]^m$ is called {\it $L$-monic} or just {\it monic}, if the coefficient of $\In(a)$ is $1$. Let $M$ be a submodule of $\K[[x]]^p$. Then $\In(M)$ denotes the initial module of $M$: $\In(M):=\langle\In(a); a\in M\rangle$. A subset $\Fst=\{f_1,\ldots,f_l\}$ of $M$ is called a {\it standard basis} of $M$ (w.r.t. $<$) if $$\In(\langle f_1,\ldots, f_l\rangle)=\langle \In(f_1), \ldots, \In(f_l)\rangle.$$

Consider a $\K[[x]]$-linear map $\ell:\K[[x]]^m\ra \K[[x]]^p; a=(a^1,\ldots, a^m) \mapsto \sum_{i=1}^mf_ia^i$. Assume that $f_1,\ldots, f_m$ form a standard basis for $\im(\ell)$. Choose a partition of $S=\textnormal{supp}(\In(\im(\ell))) \subseteq \N^n\times\{1,\ldots,p\}$ by disjoint sets $S_1,\ldots, S_m$, such that $S_i \subseteq\textnormal{supp}(\K[[x]]\cdot \In(f_i))$. For a subset $A$ of $\N^n\times \{1,\ldots, p\}$ we will consider the projection $p_1(A)$ onto the first factor $p_1(A):=\{\alpha \in \N^n; \exists j \in \{1,\ldots, p\}: (\alpha,j)\in A\}$.  Moreover, set $$\Delta(\ell):=\{ b\in \K[[x]]^p; \textnormal{supp}(b)\cap\textnormal{supp}(\In(\im(\ell)))=\emptyset\}$$ and $$\nabla(\ell):=\{a\in \K[[x]]^m; \textnormal{supp}(a_if_i)\subseteq S_i, \ \ \textnormal{for all} \ \ i\}.$$ With $\ell^0: \K[[x]]^m\ra\K[[x]]^p; a \mapsto \sum_i \In(f_i)a^i$ this means $$\K[[x]]^m=\ker(\ell^0)\oplus\nabla(\ell) \ \ \textnormal{and} \ \ \K[[x]]^p=\im(\ell^0)\oplus \Delta(\ell).$$ As the next theorem shows even more is true: This decomposition still holds with $\ell^0$ replaced by $\ell$.

\medskip

\begin{Thm} \label{divisionthm}
Let $f_1,\ldots, f_m \in \K[[x]]^p$ be a monic standard basis. Denote $\In(f_i)$ by $x^{\alpha^i,j_i}$. Let $S=\uplus_{i=1}^m S_i$ be a partition of the support of the initial module generated by $f_1,\ldots, f_m$. Then for all $f\in \K[[x]]^p$ there exist unique quotients $g_j\in \K[[x]]$, $1\leq j \leq m$, and a unique remainder $h \in \K[[x]]^p$ such that\\

1. $f=g_1f_1+\ldots + g_mf_m + h$ with $\textnormal{supp}(g_i)\subseteq p_1(S_i) - \alpha^i$ and $\textnormal{supp}(h)\subseteq \N^{n}\times \{1,\ldots, p\} \setminus S$.\\

2. The linear map $\ell:\K[[x]]^m\ra \K[[x]]^p; a=(a^1,\ldots, a^m) \mapsto \sum_i f_ia^i$ induces a direct sum decomposition $$\K[[x]]^m=\ker(\ell)\oplus \nabla(\ell) \ \ \textnormal{and} \ \ \K[[x]]^p=\im(\ell)\oplus \Delta(\ell).    $$
\end{Thm}
\medskip

For a proof of this theorem we refer to \cite{Grauert}, \cite{Hironaka}, \cite{Galligo1}, \cite{Galligo}, \cite{Schreyer} or \cite{HM}. Using the notation of the theorem, a scission $\sigma$ of $\ell$ can be constructed as follows: Denote by $D_\Fst$ the map $$D_\Fst:\im(\ell) \ra \nabla(\ell); f\mapsto (g_1,\ldots, g_m).$$ A projection onto the image of $\ell$ is given by $$\pi:\K[[x]]^p\ra \im(\ell); f\mapsto f- h.$$
Then $\sigma:= D_{\Fst}\circ \pi$ is a scission of $\ell$.\\

If $\Fst$ doesn't form a standard basis, we may choose a (monic) standard basis $\hat \Fst = \{F_1,\ldots, F_l\}$, inducing a $\K[[x]]$-linear map $\rho:\K[[x]]^l \ra \K[[x]]^m$ with $$\rho_j(a)=\sum_{i=1}^{l}\beta_j^i a^i$$ where $F_i=\sum_j\beta_j^i f_j$. Thus $\ell\circ \rho = \hat \ell$, where $\hat \ell:\K[[x]]^l\ra \K[[x]]^p; a=(a^1,\ldots, a^l)\mapsto \sum_i F_ia^i$. By the preceding remarks $\tau:=D_{\hat \Fst}\circ \pi_{\im(\hat \ell)}$ is a scission of $\hat \ell$ and thus 

\begin{equation} \label{scission}
\sigma:=\rho \tau \pi_{\im(\ell)}
\end{equation}

\smallskip

is a scission of $\ell$. We will refer to this scission of $\ell$ as the {\it scission of $\ell$ corresponding to the standard basis $\hat \Fst$}.

\smallskip

\begin{Prop}
Let $\ell$ be a $\K[[x]]$-linear map. Then the scission $\sigma$ of $\ell$ corresponding to a standard basis $\Fst$ of $\im(\ell)$ has degree of contraction $\kappa(\sigma)$ bounded by $$\kappa(\sigma)\geq - \max_{f\in \Fst} \{\w(f)\}.$$
\end{Prop}

\begin{proof}
Consider a linear map $\ell:\K[[x]]^m\ra \K[[x]]^p; a\mapsto \sum_{i=1}^m f_i a^i$ with $\Fst=\{f_1,\ldots, f_m\}$ forming a standard basis of $\im(\ell)$. Define $\sigma$ as the scission $D_\Fst\circ \pi$ of $\ell$ corresponding to $\Fst$. Further, set for $1\leq i \leq m$ $$\Delta_i=\{a\in \K[[x]]; \textnormal{supp}(a)\subseteq p_1(S_i) -\alpha^i \}$$ and $$\bar\Delta=\{a\in \K[[x]]^p; \textnormal{supp}(a)\subseteq \N^{n}\times \{1,\ldots, p\} \setminus S\}.$$ The division map $D_\Fst$ is given by (see \cite{Schreyer})

$$D_\Fst=\sum_{k=0}^\infty \Phi_1^{-1}\circ\left(\Phi_2\circ
\Phi_1^{-1}\right)^k.$$

Here, $\Phi_1$ and $\Phi_2$ are defined as

$$\Phi_1: \left(\oplus_{i=1}^m\Delta_i\right)\oplus\bar\Delta \lra \K[[x]]^p; (g_1, \ldots, g_m, h) \mapsto g_1
x^{ \alpha^1,j_1} + \ldots + g_mx^{\alpha^m,j_m}$$ and $$\Phi_2:\left(\oplus_{i=1}^m\Delta_i\right)\oplus\bar\Delta  \lra
\K[[x]]^p; (g_1, \ldots, g_m, h) \mapsto g_1 u^1 + \ldots + g_pu^p,$$
with $u^i$ given by $\ell^i = x^{\alpha^i, j_i} - u^i$. After a change of the linear form $L$ (see \cite{Schreyer}) we may assume that $\w x^{\alpha^i,j_i} < \w u^i$. Thus, $\kappa(\left(\Phi_2\circ
\Phi_1^{-1}\right)^k)\geq 0$: $\Phi_1^{-1}$ operates on a power series $a\in \K[[x]]^p$ by singling out the $x^{\alpha^i,j_i}$. Hence, $\textnormal{supp} (\Phi_1^{-1}(a))_i \cap p_1(S_i) \subseteq p_1(S_i) - \alpha^i$, $i=1,\ldots,m$. The maximal shift is given by the maximal order of the standard basis $\Fst$. The second map $\Phi_2$ multiplies each component of $\Phi_1^{-1}(a)$ by a $u^i$, which has higher order than $x^{\alpha^i,j_i}$ -- the order of $a$ has increased after applying $\Phi_2\circ \Phi_1^{-1}$, which means positive contraction degree! By section \ref{firstdefinitions} $$\kappa(D_\Fst)\geq \min_k\{\kappa(\Phi_1^{-1}\circ\left(\Phi_2\circ
\Phi_1^{-1}\right)^k)\}.$$
So $\kappa(D_{f_i})=\kappa(\Phi_1^{-1})= - \max_{f \in \Fst} \{\w f\}$. If $\Fst$ is not a standard basis, a scission $\sigma$ is constructed as in equation (\ref{scission}). Then $\kappa(\sigma)\geq \kappa(\rho) + \kappa(\tau) \geq \kappa(\tau)$. For $\tau$ the claim was just proven.
\end{proof}

We finish this section with a special $\K[[x]]$-linear mapping:\\

\begin{Constr}
Let $\ell:\K[[x]]^n\ra\K[[x]]^n$ be given by a matrix $A\in \K[[x]]^{n\times n}$ with $\det A \neq 0 \in \K[[x]]$. In this case it is sufficient to use the Weierstrass division theorem instead of Theorem \ref{divisionthm}. A scission $\sigma$ for $\ell$ will be defined by multiplication with $A^{adj}$, the adjoint matrix of $A$, and then taking the quotient by Weierstrass division. For the last step, recall the Weierstrass division theorem: Let $h \in \K[[x]]$, then there exists a linear change of coordinates $\varphi$ so that $h\varphi$ is $x_n$-regular of order $\w h$. The theorem asserts that for any $g\in \K[[x]]$ there exist {\it unique} power series $B,R \in \K[[x]]$ such that 
$$h\varphi = B\cdot g + R,$$
with $\deg_{x_n} R < \w h$. The series $B\varphi^{-1}$ is called the quotient of $h$  by $g$ using Weierstrass division, and will henceforth be denoted by $Q(h; \varphi, g)$. Note that $Q$ is textile.\\
Let $y=(y_1,\ldots, y_n)^\top$ be defined by $y =A^{adj}\cdot z$ for given $z=(z_1,\ldots, z_n)^\top \in \K[[x]]^n$. In addition assume that $\varphi_i$ are linear coordinate changes such that $y_i\varphi_i$ is $x_n$-regular of order $\w y_i$. Then we define 

$$\sigma(z)=(Q(y_1,\varphi_1, \det A),\ldots, Q(y_n,\varphi_n, \det A)).$$
\end{Constr}

\begin{Prop} \label{genericrankscission}
The map $\sigma$ as defined above is textile and a scission for $\ell$. Moreover, its degree of contraction is 
$$\kappa(\sigma) \geq - \w \det A + \kappa(A^{adj}\cdot (-)).$$
\end{Prop}

\begin{proof}
That $\sigma$ is textile follows from a proof of the Weierstrass Division theorem. An easy calculation shows that it is indeed a scission and the contraction degree follows from the results in section \ref{sectionformalmaps}.
\end{proof}

Note that the last construction differs from the one obtained earlier by means of standard basis. It's advantage lies in the relative easy computation of its degree of contraction -- no estimates on the order of a standard basis of its column module are necessary!

\subsection{Scissions of arbitrary linear textile maps}
For simplicity of notation we restrict our considerations in this section to textile linear maps on $\Aarc$. All results have generalizations to the multivariate case. An arc $a\in \Aarc$ is a vector of infinite (countable) length. In \ref{firstdefinitions} textile linear maps $\Aarc \ra \Aarc$ were identified with matrices of countable many rows and columns, called $\N^2$-matrices. By definition of textile maps, each row contains just finitely many nonzero entries. We say a matrix $S$ is a {\it scission} of a matrix $A$, if $ASA=A$ holds. First consider the case of linear maps between finite dimensional $\K$-vectorspaces. There, linear algebra provides a simple way of constructing scissions: For example let $\ell:\K^3\ra \K^2$ be the $\K$-linear map given by the matrix $$A:=\left(\begin{array}{ccc} 1 & 1 & 1 \\ 1 & 1 & 1 \end{array}\right).$$ According to the Gauss Algorithm there are matrices $P\in \textnormal{Gl}_2(\K)$ and $Q \in \textnormal{Gl}_3(\K)$ such that $A$ is equivalent to $$\tilde A=\left(\begin{array}{ccc} 1 & 0 & 0 \\ 0 & 0 & 0 \end{array}\right),$$ i.e., $PAQ=\tilde A$. $P$ operates on $A$ by elementary row operations, $Q$ by elementary column operations. A scission of $\tilde A$ is given by the matrix
$$\tilde S:=\left(\begin{array}{cc} 1 & 0 \\ 0 & 0\\ 0 & 0 \end{array}\right).$$ Indeed, $\tilde A \tilde S \tilde A = \tilde A$. Thus, the matrix $S:=Q\tilde S P$ is a scission of $\ell$. The same reasoning applies for textile linear maps $\ell: \Aarc \ra \Aarc$. Two textile linear maps $\ell_1, \ell_2:\Aarc \ra \Aarc$ are called {\it equivalent} if there exist linear automorphisms $P, Q$ of $\Aarc$ such that $$P\ell_1 Q = \ell_2.$$ Let $\ell:\Aarc \ra \Aarc$ be a textile linear map. We represent $\ell$ by an $\N^2$-matrix $\left(\ell_{ij}\right)_{ij}$ denoted by $\underline{\ell}$. Since the maps are textile for all $i \in \N$ there exists an $N_i\in  \N$ such that $\ell_{ij}=0$ for $j>N_i$. The submatrix of $\underline{\ell}$ obtained by fixing the first index $i\in \N$ is called the $i$-th row of $\underline{\ell}$. Analogously, the $j$-th column of $\underline{\ell}$ is defined as the submatrix where the second index is fixed to $j$. The $\N^2$-matrix $\underline{\ell}$ is said to be in {\it canonical form} if all rows and columns contain at most one nonzero entry. Denote by $\underline{\epsilon}^{ij}$ the matrix $(\delta_{ik}\delta_{lj})_{kl}$ with $\delta_{kl}$ the Kronecker symbol.\\
Let $(\underline{s}_l)_{l\in\N}$ be a sequence of $\N^2$-matrices. An $\N^2$-matrix $\underline{s}$ is the {\it limit} of $(\underline{s}_l)_l$ if $s$ is the limit of $(s_l)_l$ as textile linear maps (see section \ref{firstdefinitions}). This means that for all $N\in \N$ there exists an $N_0$ such that for all $l\geq N_0$ the $j$-th row of $(\underline{s}-\underline{s_l})$ is zero for $0\leq j \leq N$. In analogy to the notation for power series, we write for the last condition $\w (\underline{s} - \underline{s}_l)> N$.

\medskip

\begin{Thm} \label{canonicalform}
Every textile linear map $\ell:\Aarc \ra \Aarc$ is equivalent to a textile linear map in canonical form. The transformation maps $P,Q$ and the canonical form $\tilde \ell$ are obtained by the algorithm described below.
\end{Thm}
\medskip

\begin{Alg} $\ $

\begin{enumerate}
\item[(i)] $i:=0$, $\underline{\tilde \ell}:=\underline{\ell}$, $\underline{P}:=\underline{\id}$, $\underline{Q}:=\underline{\id}$;

\item[(ii)] Set $N_i :=\max\{j;\ell_{ij}\neq 0\}$, w.l.o.g. the {\it pivot element} $\ell_{i N_i}$ equals $1$;

\item[(iii)] Eliminate the $(l, N_i)$ entries of $\underline{\ell}$ for $l>i$ by applying $$\underline{P}_i=\underline{\id} - \sum_{l=1}^{\infty} \ell_{i+l,N_i}\cdot \underline{\epsilon}^{i+l,i}$$ from the left;

\item[(iv)] Eliminate the $(i,k)$ entries of $\underline{\ell}$ for $k< N_i$ by applying $$\underline{Q}_i :=\underline{\id} - \sum_{j=0}^{N_i-1}\ell_{i,j}\cdot\underline{\epsilon}^{N_i,j}$$ from the right;

\item[(v)] Set $\underline{P}:= \underline{P}_i\circ \underline{P}$, $\underline{Q}:= \underline{Q}\circ \underline{Q}_i$, $\underline{\tilde \ell}:=\underline{P} \circ \ell \circ \underline{Q}$ and $i=i+1$. Return to (ii).
\end{enumerate}
\end{Alg}
\medskip

\begin{Rem}
Note that the constructed $\N^2$-matrices $\underline{P}$ (resp. $\underline{Q}_i$) are lower diagonal $\N^2$-matrices. Hence the corresponding linear maps $P$ (resp. $Q$) are invertible.
\end{Rem}
\smallskip

\begin{Ex}
Consider the difference operator $\Delta = s^d + h_{d-1}s^{d-1} + \ldots + h_0 \in \K[s]$, $s\circ (a_0,a_1,\ldots) = (a_1,\ldots, a_2,\ldots)$, with induced map $\ell: \Aarc \ra \Aarc$, $$a \mapsto \ell(a) = ( h_0a_0 + \cdots + h_{d-1}a_{d-1} + a_d, h_0a_1 + \cdots + h_{d-1}a_d + a_{d+1}, \ldots ),$$

for $a\in \Aarc$. The algorithm yields as the canonical form of $\ell$

$$a\mapsto \tilde \ell (a)= (a_d,a_{d+1}, a_{d+2}, \ldots ,a_{d+i}, \ldots).$$

This implies the well known fact that the set of solutions to the equation $\Delta \circ a = b$ is isomorphic to $\K^d$.
\end{Ex}

\medskip

\begin{proof} 
The finite composition $\underline{P}_j\circ \cdots \circ \underline{P}_0$ is denoted by $\underline{P}^j$ and analogously $\underline{Q}^j:=\underline{Q}_0\circ \cdots \circ \underline{Q}_j$. The sequence $(P^j)_{j\in \N}$ converges to an $\N^2$-matrix $\underline{P}$, since $$\w ( \underline{P_j} - \underline{\id})\geq j.$$ Thus, $\w( \underline{P}^{j+1} - \underline{P}^j) \geq j+1$. All $\underline{P_j}$ are lower diagonal, hence, the same is true for $\underline{P}$. For $(\underline{Q}^j)_j$ it is sufficient to show that for all $l\in \N$ there exists an  $n_l$ such that for all $m,n> n_l$ $$\w (\underline{Q}^m - \underline{Q}^n) > l.$$
Fix an $l\in \N$ and define $n_l$ as a natural number such that all $j\leq l$ appear as the column element of a pivot element $\ell_{rj}, r<n_l$. This gives convergence of $(\underline{Q}^j)_j$. The algorithm obviously transforms $\underline{\ell}$ into its canonical form.
\end{proof}
\medskip

\begin{Constr}  \label{existenceofscissions}
A scission $\sigma$ of a $\K$-linear map $\ell: \Aarc \ra \Aarc$ can be constructed as follows:
First one reduces to the case of linear maps in canonical form. Assume a scission $\sigma_{\tilde \ell}$ is constructed for the canonical form $\tilde \ell$ of $\ell$. A scission of $\ell$ is obtained by $\sigma_\ell := Q\sigma_{\tilde \ell} P$, where $P$, $Q$ are automorphisms of $\Aarc$ with $P\ell Q = \tilde \ell$.\\
Thus, without loss of generality $\ell$ has matrix representation $\underline{\ell}$ in canonical form. Denote by $I_1$ the set of indices $i\in \N$ s.t. the $i$-th row of $\underline{\ell}$ is nonzero and by $I_2$ the set of indices $j\in \N$ such that the $j$-th column is nonzero. Then $$\im(\ell)=\{\gamma \in \K^\N; \gamma_k=0 \ \ \textnormal{for all} \ \ k\in \N\setminus I_1\}$$ and $$\ker(\ell)=\{\gamma\in \K^\N; \gamma_k = 0 \ \ \textnormal{for all} \ \ k \in I_2\}.$$ Then $(\ell|_\Larc)^{-1} \circ \pi$ is a scission of $\ell$, where $\pi:\K^\N \ra  \im(\ell)$ is the natural projection and $\Larc:=\{\gamma \in \Aarc; \gamma_k = 0 \ \ \textnormal{for} \ \ k\in \N\setminus I_2\}$ is a direct complement of $\ker(\ell)$ as a $\K$-vector space.
\end{Constr}

\bigskip




\section{The Rank Theorem} \label{sectionrankthm}

Consider a textile map $f\colon \Ce^m \ra \Ce^p$ with tangent map $\ell=T_{a_0}f$ at $a_0$. Our purpose is to linearize  $f$ locally at $a_0$  by applying local automorphisms of source and target. 
This is obtained by generalizing the Rank Theorem in \cite{HM} (proven there for tactile maps) in three directions: We allow textile maps, we give a version with parameters, and also a statement with nilpotent coefficients. The main argument of the proof from \cite{HM} is recalled and extended to textile maps in section \ref{sec:Rank_Theorem_Without_Parameters}. We introduce textile maps depending on parameters in \ref{sec:Parametric_Rank_Theorem} and prove the corresponding theorem. The nilpotent case is treated in section \ref{sec:RankTheorem_for_TestRings}.\\

\subsection{The Rank Theorem for textile maps}\label{sec:Rank_Theorem_Without_Parameters}

We start with some technical preliminaries. The notation used here is abutted to the analytic case treated in \cite{HM}. A textile map $\gamma: U \subseteq \Ce \ra \Ce^m$ is called a {\it curve} (over $U$) in $\Ce^m$. The argument of $\gamma$ is called the parameter of the curve and is mostly denoted by $s$. For an arbitrary subvector space $J \subseteq \Ce^m$ a textile map $\gamma\colon U\subseteq \Ce\ra \Ce^m$ is called a {\it curve in $J$} if $\im(\gamma)\subseteq J$. The set of all curves in $J$ is denoted by $\Gamma(J)$. The tangent space $T_aJ$ of $J$ at $a \in J$ is defined as the $\K$-vector space with elements $$\p\gamma(0),$$ for a curve $\gamma$ in $J$ with $\gamma(0)=a$. Obviously $\textnormal{T}_a\Ce^m$ can be identified with $\Ce^m$ for any $a \in \Ce^m$. Thus, a curve in the tangent bundle of $\Ce^m$ is given by a pair of curves $(\gamma, b) \in \Gamma(\Ce^m)^2$.\\

A textile map $f:\Ce^m \ra \Ce^p$ induces an appropriate tangent map $Tf: \textnormal{T}\Ce^m \ra \textnormal{T}\Ce^p$ and further by composition a map $$\Gamma(\textnormal{T}\Ce^m) \ra \Gamma(\textnormal{T}\Ce^p); (\gamma, b) \mapsto (f\gamma, T_\gamma f\cdot b).$$
In the case of tactiles $T_\gamma f\cdot b$ is given by $$T_\gamma f\cdot b = \sum_{i=1}^n\left(\p_i f\right)(\gamma(s))\cdot b^i(s).$$

Let $U\subseteq \Ce^m$ be an $\m$-adic neighbourhood of $a_0 \in \Ce^m$. Assume that the image $\im(T_{a_0}f)$ has (as a topological vector space)  a direct complement $J$, i.e.  $$\im(T_{a_0}f)\oplus J = \Ce^p.$$  Moreover let $\gamma$ be a curve in $U$ with $\gamma(0)=a_0$; then we say $f$ has {\it constant rank at $a_0$ with respect to $J$} if for all curves of the form $(f\gamma, \eta)$ in $\textnormal{T}\Ce^p$ there exist unique curves $\rho$ and $\tau$ such that  $(\gamma, \rho)$ defines a curve in $\textnormal{T}U$ and $\tau \in \Gamma(\textnormal{T}J)$ with the property 

\begin{equation} \label{rankcondition}
\eta = T_\gamma f \cdot \rho + \tau.
\end{equation}
\medskip

A textile map $f:U\subseteq \Ce^m \ra \Ce^p$ is called {\it flat at $a_0$} if for
all curves $\gamma$ in $U$ with $\gamma(0)=a_0$ the evaluation at $s=0$ defines a surjective map
$$\ker(T_\gamma f: \textnormal{T}_\gamma U\ra \textnormal{T}_{f\gamma} \Ce^p) \lra \ker(T_{a_0}f: \textnormal{T}_{a_0}U \ra \textnormal{T}_{f(a_0)} \Ce^p).$$

\medskip

Similar to [HM] one proves the following proposition (see \cite{bruschek}):

\begin{Prop}\label{flatisrank}
Let $f:U\subseteq \Ce^m \ra \Ce^p$ be a textile map of type $f=\ell + h$. Assume there exists a scission of $\ell$, such that $\sigma h$ is contractive. Then $f$ has constant rank at $0$ if and only if $f$ is flat at $0$.
\end{Prop}
\smallskip

\begin{Rem}
The contractivity condition on $\sigma$ and $h$ is necessary to conclude that flatness means constant rank. The other direction doesn't need the construction of a scission of $\ell$. A typical class of flat maps is given by maps with injective tangent map.
\end{Rem}

\begin{samepage}
\begin{Thm}[Rank Theorem] \label{rankthm}
Let $f: U \subseteq \Ce^m \ra \Ce^p$ be a textile map on an $\m$-adic neighbourhood $U$ of $0$ with $f(0)=0$. Write $f$ in the form
$f=\ell + h$ with $\ell=T_0f$ the tangent map of $f$ at $0$. Assume that $\ell$ admits a scission
$\sigma$ for which $\sigma h$ is contractive on $U$, and that $f$ has constant rank at $0$
with respect to $\ker(\ell\sigma)$. Then there exist locally invertible textile
maps $u:U\ra \Ce^m$ and $v:f(U) \ra \Ce^p$ at $0$ such that
$$vfu^{-1} = \ell.$$
\end{Thm}
\end{samepage}

\begin{proof}
The proof naturally falls into two parts. First construct two
invertible textile maps $u,v$; second show that these maps indeed
linearize $f$. This will be ensured by the rank condition. Set $u=
\id_{\Ce^m} + \sigma h$. Obviously $u(0)=0$ and $T_0 u = \id$. By
contractiveness of $\sigma h$ on $U$ this $u$ is invertible on $U$. Note:
$$\ell \sigma f = \ell \sigma (\ell + h) = \ell + \ell \sigma h =
\ell (\id + \sigma h) = \ell u.$$ Thus, $fu^{-1}$ fulfills $\ell \sigma f u^{-1} = \ell$.
 Without loss of generality one can assume for the rest of the proof that $\ell \sigma f = \ell$.
 The second map $v$ is defined by $v= \id_{\Ce^p} - (\id_{\Ce^p} -
\ell\sigma)f\sigma \ell \sigma$. Again: $v(0)=0$ and $v$ is invertible with inverse $v^{-1}=\id_{\Ce^p} +
(\id_{\Ce^p} - \ell \sigma)f\sigma \ell \sigma$. For later use we write 

\begin{eqnarray*}
v & = & \id - f\sigma \ell \sigma + \ell  \sigma f \sigma \ell
\sigma\\
& = & \id - f\sigma \ell \sigma + \ell \sigma\\
& = & \id - h \sigma,
\end{eqnarray*}

with $\kappa(h\sigma)>0$. Note: The diffeomorphisms $u$ and $v$ are constructed without using the rank condition at all. Assume $f$ fulfills the equation $(*)$: $f = f\sigma \ell$. This would imply that

\begin{eqnarray*}
vf & = & f - (\id - \ell \sigma) f \sigma \ell \sigma f\\
& = & f - (\id - \ell \sigma) f\sigma \ell\\
& = & f - (\id - \ell \sigma)f\\
& = &\ell \sigma f\\
& = & \ell.
\end{eqnarray*}
\smallskip

Equation $(*)$ is a consequence of the Rank Condition. The
composition $\sigma\ell$ is a projection onto the complement of $\ker(\ell)$ in $\Ce^m$. It is sufficient to show that $T_a f | _{\ker(\ell)}=0$ for
all $a\in \Ce^m$ near $0$. From $\ell \sigma f = \ell$ follows $f =
\ell + (\id_{\Ce^p} - \ell \sigma)f$ and subsequently $$T_a f = \ell
+ (\id_{\Ce^p} - \ell \sigma)T_a f.$$ Let $a$ be fixed in $U$. Define a
curve $\gamma(s):= s\cdot a$. For $b \in \ker(\ell)$
$$T_{\gamma(s)} f\cdot b = (\id_{\Ce^p} - \ell \sigma)T_{\gamma(s)}f \cdot
b$$ holds. By uniqueness of the decomposition in equation (\ref{rankcondition})

$$T_{\gamma(s)} f \cdot b=0$$ and thus $T_a f| _{\ker(\ell)}=0$.
The assertion follows.
\end{proof}

A textile map $f=\ell + h$ is called a {\it quasi-submersion} if $$\im(h)\subseteq \im(\ell).$$ Especially, if $\ell$ is surjective, $f$ is a {\it submersion}. A quasi-submersive map $f$ is linearizable if the {\it order condition} (i.e., there exists a scission $\sigma$ for $\ell$ such that $\kappa(\sigma h)>0$) is fulfilled. Indeed, by construction of $v$ in the proof of Theorem \ref{rankthm} we see that in the case of quasi-submersive $f$ this map $v$ equals the identity. Therefore, checking the Rank Condition becomes obsolete. Quasi-submersions play an important role in the applications in section \ref{sec:Applications}. In fact, the only application in section \ref{sec:Applications} where quasi-submersions do not appear can be found in section \ref{sectionmirlamel} (in that case the tangent map is injective).

\begin{Ex}
  We give an example of a textile map which is neither a quasi-submersion nor has injective tangent map (but fulfills the order condition):
  $$\m\cdot \K[[x,y]]^2\ra \m^3\subseteq \K[[x,y]]; (a,b)\mapsto xya + y^2b + x^2ab.$$
\end{Ex}

{\it Some simplification}: The Rank Condition is technical and
may be hard to check. In some applications the following simpler
condition is sufficient. Let $f: U\subseteq \Ce^m \ra \Ce^p$ be a textile map. Denote by $J$ a
direct complement of $\im(T_0f)$. Then $f$ is said to have {\it
pointwise constant rank at $0$ with respect to $J$} if for all $a\in U$
$$\im(T_af)\oplus J =\Ce^p.$$

\begin{Prop} \label{pointwiserank}
Let $f: U\subseteq \Ce^m \ra \Ce^p$ be a tactile, $U$ an $\m$-adic
neighbourhood of $0$. If for all $a\in U$ $$\In(\im (T_af))= \In(\im (T_0f))$$
then $f$ has pointwise constant rank at $0$ with respect to $$\Delta(T_0f)=\{
b\in \Ce^p; \textnormal{supp}(b)\cap\In(\im(T_0f))=\emptyset\}.$$ 
\end{Prop}

\medskip

\subsection{The Parametric Rank Theorem}\label{sec:Parametric_Rank_Theorem}

The Rank Theorem ensures linearization of textile maps $f: U\subseteq \Ce^m \ra \Ce^p$ of constant rank near a fixed cord $c\in U$. The Parametric Rank Theorem will deal with families of textile maps, thus providing a tool to trivialize vector bundles (see section \ref{sectionarcspace}).\\

Let $\Ss \subseteq \A_\K^q$ be constructible and $Y\subseteq \Ce^m$ a subfelt. A map 

$$f_\Ss: \Ss \times Y \ra \K; (s,c) \mapsto  f_{\Ss}(s,c)$$

is called a {\it regular $\K$-valued family of functions over \Ss} if there is a function $g:U\times Y \ra \K$, where $U \subseteq \A^q_\K$ containing $\Ss$ is a subvariety and $g$ locally looks like a quotient $\frac{h}{l}$ of $h,l \in \K[s_1,\ldots, s_q]\otimes_\K \K[\underline{\N}_m^{n}]$, $l$ locally non-zero, such that $g|_{\Ss\times Y} = f_\Ss$. More generally we define a {\it regular family of textile maps over $\Ss$ on $Y$} as a map

$$f_{\Ss}: \Ss\times Y \ra \Ce^p; (s,c) \mapsto  \left(f_{\Ss,\alpha}(s,c)\right)_{\alpha \in \N^{n+1}} $$ 

where each $f_{\Ss,\alpha}$ is a regular $\K$-valued family over $\Ss$. Thus for fixed $s = (s_1,\ldots, s_q) \in \Ss$ the restricted map 

$$f_s:=f_{\Ss}(s,-): Y \ra \Ce^p$$ 

is a (regular) textile map. The $s_i$ will be referred to as {\it parameters} and $\Ss$ as the {\it parameter space}. If $f_\Ss$ is a family such that all $f_s$ are linear maps, we call it a {\it linear family over $\Ss$}. As a special case the family $\id_\Ss: \Ss\times Y \ra Y$, where $\id_s = \id$ for all $s\in \Ss$, appears.\\

If $f_\Ss: \Ss\times \Ce^m \ra \Ce^p$ and $g_\Ss: \Ss\times \Ce^q \ra \Ce^m$ are two regular families we define their composition as the family $\Ss \times \Ce^q \ra \Ce^p$ given by
$$f_\Ss \circ g_\Ss:= f_\Ss \circ (proj_1,g_\Ss),$$
where $proj_1: \Ss\times \Ce^q \ra \Ss$ is the projection onto the first factor. Thus $(f_\Ss \circ g_\Ss)_s = f_s\circ g_s$. A regular family $f_\Ss: \Ss \times \Ce^m \ra \Ce^p$ is called {\it invertible} if there exists a {\it regular} family $g_\Ss: \Ss \times \Ce^p \ra \Ce^m$ over $\Ss$ with $f_\Ss \circ g_\Ss=\id$.\\

\begin{Ex}
Consider 

$$\ell_{\A_\K}: \A_\K \times \Aarc \ra \Aarc; (s,c) \mapsto (0,s c_0, sc_1,\ldots).$$ Obviously $\ell_{\A_\K}$ is a regular family. The family $\sigma_{\Ss}$, $\Ss=\A_\K\setminus \{0\}$, defined by

$$\sigma_{\Ss}: \Ss\times \m\cdot \Aarc \ra \Aarc; c \mapsto \frac{1}{s}\cdot (c_1,c_2,\ldots)$$

is regular with (right) inverse $\ell|_\Ss$. Thus $\sigma_{\Ss}$ is an invertible family.
\end{Ex}

The defining property for contractive textile maps naturally generalizes to families of textile maps over $\Ss$. Let $U=\m^l\cdot \Ce^m$ be an $\m$-adic neighbourhood of $0$. A family $f_\Ss: \Ss\times U \ra \Ce^p$ is called {\it contractive on $U$} if it fulfills condition (\ref{contraction}) independently of $\Ss$, or equivalently, it fulfills (\ref{contraction}) pointwise: $f_s$ is contractive on $U$ for all $s\in \Ss$.

\begin{Thm} \label{inversemappingthm_parametric}
Let $f_\Ss:\Ss \times U \ra \Ce^m$, $U=\m^l\cdot \Ce^m$, an $\m$-adic neighbourhood of $0$, be a regular family over $\Ss$ with $f_\Ss(-,0)=0$. Assume $f_\Ss=\ell_\Ss + h_\Ss$ with $\ell_\Ss:\Ss\times U\ra W$ linear with inverse $\tilde \ell_\Ss: \Ss\times W\ra U$. If $\tilde \ell_\Ss\circ h_\Ss $ is contractive on $U$, then $f_\Ss$ admits on $U$ an (right) inverse, i.e., there is a regular family $g_\Ss:\Ss\times f(U) \ra \Ce^m$ with $f_\Ss \circ g_\Ss =\id_\Ss$. 
\end{Thm}

\begin{proof}
The same reasoning as in the proof of Theorem \ref{inversemappingthm} can be applied in the present situation. The inverse for $f_\Ss$ is constructed as the limit of the sequence defined by the recursion $g_{\Ss}^0=0$,

$$g_{\Ss}^{i+1} = \id_\Ss - h_\Ss\circ g_{\Ss}^i.$$

By induction one shows that $g_\Ss^i$ is a regular family, the case $i=0$ being trivial. For $i>0$: The map $h_\Ss$ is a regular family on $U=\m^l\cdot \Ce^m$. Therefore -- using the induction hypotheses -- we see that $h_\Ss \circ g_\Ss^{i-1}$ is regular and so is $g_\Ss^i$. Moreover, $(\w (g_\Ss^{i+1}- g_\Ss^i))_i$ is an increasing sequence. Thus the limit of $(g_\Ss^i)_i$ exists and the $\alpha$-coefficient of $g_\Ss$, $\alpha \in \N^{n+1}$, is the same as the $\alpha$-coefficient of $g_\Ss^{i_\alpha}$ for some $i_\alpha \in \N$, for which we have just shown that it is a regular family.
\end{proof}

Consider a regular family $f_\Ss$ on $U\subseteq \Ce^m$. The {\it tangent map} with respect to $\Ce^m$ of the family $f_\Ss$ (see section \ref{differentialcalc}) is the regular family $Tf_\Ss$ given by linear maps $Tf_s$ which are the tangent maps to the $f_s$, $s\in \Ss$.\\
Let $\ell_\Ss: \Ss\times U \ra \Ce^p$, $U \subseteq \Ce^m$, be a regular linear family. A {\it scission} for $\ell_\Ss$ is a regular linear family $\sigma_\Ss:\Ss \times \Ce^p \ra \Ce^m$ such that $\ell_\Ss \sigma_\Ss \ell_\Ss = \ell_\Ss$.\\
Assume that $J$ is a direct complement of $\im(T_{a_0} f_{s})$ in $\Ce^p$, for all $s \in \Ss$, where $a_0 \in U$. The family $f_\Ss$ is {\it of constant rank at} $a_0$ {\it with respect to} $J$ (or {\it fulfills the parametric rank condition at $a_0$ w.r.t. $J$}) if for all $s\in \Ss$ the restricted maps $f_s$ are of constant rank at $a_0$ with respect to $J$ (see section \ref{sectionrankthm}).

\begin{Thm}[Parametric Rank Theorem] \label{parametricrankthm}
Let $f_\Ss: \Ss\times U \ra \Ce^p$ be a regular family, $U=\m^l\cdot \Ce^m$; write $f_\Ss = \ell_\Ss +  h_\Ss$ with $\ell_\Ss := T_0 f_\Ss$. Assume $f_\Ss(-,0)=0$ and that the following conditions hold: (i) the linear family $\ell_\Ss$ admits a scission $\sigma_\Ss$ s.t. $\sigma_\Ss h_\Ss$ is contractive; (ii) the family $f_\Ss$ fulfills the parametric rank condition w.r.t. a direct complement of $\im(T_0f_\Ss)$. Then there exist regular (locally) invertible families $u_\Ss$ and $v_\Ss$ with $$v_\Ss f_\Ss u_\Ss^{-1} =\ell_\Ss.$$
\end{Thm}

\bigskip

\begin{proof}
The arguments used in the proof of Theorem \ref{rankthm} may be applied here. By the assumptions and Theorem \ref{inversemappingthm_parametric} the families  $u_\Ss=\id_\Ss + \sigma_\Ss h_\Ss$ respectively $v_\Ss=\id_\Ss - (\id_\Ss - \ell_\Ss \sigma_\Ss)f_\Ss\sigma_\Ss \ell_\Ss \sigma_\Ss$ are regular and invertible. By condition (ii) and Theorem \ref{rankthm} these families linearize $f_\Ss$ pointwise, thus

$$v_\Ss f_\Ss u_\Ss^{-1} = \ell_\Ss.$$ 
\end{proof}

\subsection{The Rank Theorem for test-rings}\label{sec:RankTheorem_for_TestRings}
Let $A$ be {\it test-ring}, i.e., a local, commutative, and unital $\K$-algebra with nilpotent maximal ideal $\n$. Denote by $\Ce_A$ the space of $n$-cords with coefficients in $A$. By  Theorem \ref{inversemappingthm} it follows that for a textile map $h\colon U\subseteq \Ce_A^m\ra \Ce_A^m$ which is contractive with respect to the refined order $\W$, the map $\id + h$ is invertible. Again $U$ denotes an $\m$-adic neighbourhood of $0$. The proof of Theorem \ref{rankthm} can be adopted to this setting, thus giving the analogous assertion:

\begin{Thm}[Rank Theorem for test-rings]\label{thm:RankTheorem_for_TestRings}
Let $A$ be a test-ring and $f: U \subseteq \Ce_A^m \ra \Ce_A^p$ be a textile map on an $\m$-adic neighbourhood $U$ of $0$ with $f(0)=0$. Write $f$ in the form
$f=\ell + h$ with $\ell=T_0f$ the tangent map of $f$ at $0$. Assume that $\ell$ admits a scission
$\sigma$ for which $\sigma h$ is contractive on $U$, and that $f$ has constant rank at $0$
with respect to $\ker(\ell\sigma)$. Then there exist locally invertible textile
maps $u:U\ra \Ce_A^m$ and $v:f(U) \ra \Ce_A^p$ at $0$ such that
$$vfu^{-1} = \ell.$$
\end{Thm}
\bigskip





\section{Six instances of the linearization principle}\label{sec:Applications}

We are now ready to prove by one argument the results in singularity theory mentioned in the introduction. Recall that $\Ce$ and $\Aarc$ denote the $\K$-algebras of $n$ respectively $1$ cords. In the latter case we speak of arcs. They naturally identify with elements of $\K[[x_1,\ldots, x_n]]$ respectively $\K[[t]]$.

\subsection{On a fibration theorem by Denef and Loeser} \label{sectionarcspace}
Let $X \subseteq \A_{\K}^n$ be an affine variety over a field $\K$ of characteristic $0$ defined by elements $f_1, \ldots, f_p$ of $\K[x_1,\ldots, x_n]$. These polynomials define by substitution of power series $a^i\in \K[[t]]$ for the $x_i$ a tactile map $F:\Aarc^n \ra \Aarc^p$ with coefficients $(F_i)_{i\in \N}$. The felt defined by $F$ is called the {\it arc space of $X$} and denoted by $X_\infty$. In terms of power series $$X_\infty =\{a=(a^1,\ldots, a^n) \in \K[[t]]^n; f_1(a)=\ldots= f_p(a)=0\}$$
For $q\in \N$ we write $\pi_q$ for the natural projection $X_\infty \ra \left(\Aarc/{\m^{q+1}}\right)^n \cong (\K^{q+1})^n$ sending an arc $a$ to the vector of its first $q+1$ coefficient-vectors $(a_0,\ldots, a_{q})$. The image of $X_\infty$ under $\pi_q$ is contained in $$X_q=\{ a=(a_0,\ldots, a_{q}) \in (\K^n)^{q+1}; F_0(a)=\ldots=F_{q}(a)=0\},$$ the set of ``approximate solutions'' called the {\it $q$-th jet-space of $X$}. Note: $\pi_q(X_\infty)$ consists of those approximate solutions which can be lifted to (exact) solutions in $\Aarc^n$. For an element $g\in \K[x_1,\ldots, x_n]$ and an arc $a \in \Aarc^n$ we define the {\it order of $a$ with respect to $g$} as $$\w_g(a) :=\w(g(a)).$$ Here $g(a)$ is obtained by substituting $a^i$ for $x_i$ in $g$. Let $I \subseteq \K[x_1,\ldots, x_n]$ be an ideal. The {\it order of $a$ with respect to $I$} is defined as $$\w_I (a) := \min\{\w_g a ; g \in I\}.$$ Especially if $I$ denotes the ideal of $\Sing (X)$, we will write $\Aarce(X)$, $e\in \N$, for the set of arcs on $X$ lying in $\Sing (X)$ with order at most $e$, i.e., $$\Aarce(X)=\{a \in X_\infty; \w_I(a)\leq e\}.$$\\

\begin{Thm}\label{localtrivialitytheorem}
(\cite{loeser}, Lemma 4.1) Let $X$ be an affine algebraic variety over $\K$ of pure dimension $d$; let $e\in \N$. The map $$\pi_e: X_\infty \ra \pi_e(X_\infty)\subseteq (\K^n)^{e+1}$$ is a piecewise trivial fibration over $\pi_{e}(\Aarce(X))$ with fibre $\m^{e+1}\cdot {\Aarc^d}$. This means: $\pi_{e}(\Aarce(X))$ can be covered by finitely many locally (Zariski-) closed subsets $\Ss$ such that there exist textile invertible maps $\varphi$ with commutative diagram

$$\xymatrix@1{ \pi_e^{-1}(\Ss) \ \ar@{->}[rr]^{\varphi} \ar@{->}[rdd]_{\pi_e} & & \ \Ss\times (\m^{e+1}\cdot \Aarc^d) \ar@{->}[ldd]^ {\textnormal{proj}_1}\\
& & \\
& \Ss &  }$$

\end{Thm}
\medskip

This result implies Lemma 4.1 in \cite{loeser}: the trivializing textile map $\varphi$ is by construction (see below) compatible with the truncation maps $\pi_q$, respectively $\pi^q_p$, $q\geq p$, as described in section \ref{firstdefinitions}, thus inducing a trivial fibration $\pi_{q+1}(X_\infty) \ra \pi_q(X_\infty)$, $q\geq e$, over $\pi_q(\Aarce(X))$ (with respective affine bundle structure) in the sense of \cite{loeser}. The proof of Theorem \ref{localtrivialitytheorem} will be given in the next two sections.\\

The proof of the trivialization of $\pi_e$ in \cite{loeser} and \cite{Looijenga_Motivic_Measures} makes use of Hensel's Lemma (see e.g. Corollary 1 in III, \S 4.5 of \cite{Bourcomalg}), which in turn follows from Thm. 2, III, \S 4.5 there. Notice that the statement of this Thm. 2 is analogous to the Rank Theorem in the respective setting. Our proof of Thm. \ref{localtrivialitytheorem} gives in addition an
explicit description of the trivializing map $\varphi$.\\

\subsubsection*{The Hypersurface Case} \label{sectionhypersurfacecase}

Let $X=V(f)\subseteq \A_{\K}^n$ with $f\in \K[x_1,\ldots, x_n]$ be a (reduced) hypersurface. The singular locus $\Sing(X)$ is given by the Jacobian ideal $$J_f=\langle f,\p_1 f, \ldots, \p_nf \rangle.$$
Each $a \in \Aarce(X)$ can be decomposed as $a= \bar a + \hat a$ where $\bar a$ denotes the arc $(a_0,\ldots, a_{e},0,\ldots)$ and $\hat a=a-\bar a$. Thus, $\w(\hat a)\geq e+1$. We identify $\bar a$ with the element $\pi_e(a)$. The polynomial $f$ induces a tactile $\Aarc^n\ra \Aarc$. 
We view $f$ as a regular family over the parameter set $\Ss=\pi_e(\Aarce(X))$, which is constructible, see e.g. \cite{reguera_families}, \S 1, or section (4.4) in \cite{loeser}. The parameter set admits a partition into locally closed subsets $\Ss_{i,e'}$ (w.r.t. the Zariski-topology on $\A^{n(e+1)}_\K$), $(i,e') \in \Lambda:= \{1,\ldots,n\}\times \{0,\ldots, e\}$, defined as 
$$\Ss_{i,e'}:=\{a \in \pi_e(\Aarce(X)); \min_j\{\w(\p_jf(a))\} = \w(\p_if(a))=e'\}.$$

Consider for $\alpha \in \Lambda$ the family $$f_{\Ss_\alpha}: \Ss_\alpha \times \m^{e+1}\cdot \Aarc^n \ra \Aarc; (\bar a,\eta)\mapsto f(\bar a+\eta).$$ For simplicity of notation we will denote this family again by $f$. Taylor expansion yields $$f(\bar a + \eta)=f(\bar a) + \widehat f(\bar a, \eta), $$ where $\widehat f(\bar a, \eta) = \sum\p_jf(\bar a)\cdot \eta_j + h(\bar a,\eta)$ defines a regular family on $\Ss$. The map $h$ is of order at least $2$ in $\eta$.\\

The map $\widehat f$ fulfills all assumptions of the Parametric Rank Theorem. Its linear part $\ell:=T_0 \widehat f$ has image $$\im(\ell(\bar a,-)) = \langle \p_if(\bar a)\rangle = \m^{e'+e+1}$$ for all $\bar a \in \Ss$. The image has direct complement 
$J=\oplus_{i=0}^{e'+e}\K t^i$.\\


Moreover, $$T_\varrho \widehat f \cdot \eta = \sum_j (\p_j f(\bar a) + \p_j h(\bar a,\varrho))\cdot \eta_j$$ and $\w(\p_if(\bar a) + \p_ih(\bar a,\varrho))=\w(\p_if(\bar a))$ for all $\varrho \in \m^{e+1}\cdot \Aarc^n$. Hence, by Proposition \ref{pointwiserank}, $\widehat f$ fulfills the Rank Condition for all $\bar a \in \Ss$ on $\m^{e+1}\cdot \Aarc^n$ (with respect to $J$). By Construction \ref{existenceofscissions}, the linear part $\ell$ admits a scission $\sigma$ with contraction degree $\kappa(\sigma)=-e'$ (and complement $\ker \ell\sigma=J$). Indeed, $\sigma$ is constructed by taking the quotient when dividing through $\p_if(\bar a)$ (cf. section \ref{sectionscission}). Since the division theorem (theorem \ref{divisionthm}) requires monic divisors, the coefficients of the family $\sigma=\sigma_\Ss$ will be rational in the coefficients of $\bar a$ but regular on the chosen $\Ss$. It is easy to see that $$\kappa(h)\geq 2({e+1})-({e+1}) = e + 1 > e'.$$ Hence, $\sigma h$ is contractive. The Parametric Rank Theorem gives regular invertible families $u_\Ss$ and $v_\Ss$ such that $$v_\Ss \widehat f_\Ss u_{\Ss}^{-1}=\ell_\Ss.$$
The map $\m^{e+1}\cdot \Aarc^n \ra \m^{e+1}\cdot \Aarc^n; \eta \mapsto u_{\bar a}^{-1}(\eta)$ defines an isomorphism between solutions $x$ of $$f(\bar a + x) = 0$$ and solutions $y$ of $$\ell(y)=v(-f(\bar a)).$$ By linear algebra the set of solutions to the last equation is isomorphic to $\m^{e+1}\cdot \Aarc^{n-1} $ via a linear map $\psi:\m^{e+1}\cdot \Aarc^n \ra \m^{e+1}\cdot \Aarc^{n-1}$. This results in the commutative diagram

$$\xymatrix@1{ \pi_e^{-1}(\Ss) \ \ar@{->}[rr]^{\varphi} \ar@{->}[rdd]_{\pi_e} & & \ \Ss\times \m^{e+1}\cdot \Aarc^{n-1} \ar@{->}[ldd]^ {\textnormal{proj}_1}\\
& & \\
& \Ss &  }$$

where $\varphi$ is the invertible textile map defined by $\varphi(\bar a + \hat a)=(\bar a, \psi u_{\bar a}^{-1}(\hat a))$.\\

Consider now the general situation $\Ss=\cup_{\alpha\in \Lambda}\Ss_\alpha$. Then for each $\alpha \in \Lambda$ there exists a trivializing map $\varphi_\alpha$. It is a simple matter to check that the transition maps $\varphi_\alpha \varphi_\beta^{-1}$ are linear isomorphisms for $\alpha, \beta \in \Lambda$. This completes the proof in the hypersurface case.



\subsubsection*{The General Case}
Let $X$ be an algebraic variety over $\K$ of pure dimension $d$ defined by the ideal $\langle f_1, \dots, f_N\rangle$ in $\K[x_1,\ldots, x_n]$. Its singular locus is given by the ideal $\langle f_1,\ldots, f_N, \theta\rangle$ where $\theta$ is running through all $(N-d)$-minors of $\p f= \left(\frac{\p f_i}{\p x_j}\right)_{ij}$. In this case $\Aarce(X)$ is the set of arcs $a\in X_\infty$ for which $\w \theta(a) \leq e$ for at least one $(N-d)$-minor $\theta$.\\
We proceed analogously to section \ref{sectionhypersurfacecase} and first show that pointwise linearization is possible. Let $\bar a$ again be the truncation of an arc $a\in \Aarce(X)$ to order $e$ and consider $\p f(\bar a) \in \K[[t]]^{N\times n}$. Since $\K[[t]]$ is a principal ideal domain, we can find linear automorphisms $P$, $Q$ of $\K[[t]]^N$ respectively $\K[[t]]^n$ such that $P \p f(\bar a) Q$ is in Smith-form, i.e., 

\begin{displaymath}
\left( \begin{tabular}{cccccc}
$t^{\varepsilon_1}E_1$ & $0$ & $\cdots$ & $0 $ &$\cdots$ & $0$ \\
$0$ & $\ddots$ & &$\vdots$ & & $\vdots$\\
$\vdots$ & $\ $ & $t^{\varepsilon_{N-d}}E_{N-d}$ & $\vdots $ &  & $\vdots$ \\
$0$ & $\cdots$ & $\cdots$ & $0$ & $\cdots$ &$0$\\
$\vdots$ & &&&& $\vdots$\\
$0$ & $\cdots$ &$\cdots$ & $\cdots$ & $\cdots$ & $0$
\end{tabular} \right)
\end{displaymath}

with $E_i \in \K[[t]]^{\times}$ and $\varepsilon_1\leq \ldots \leq \varepsilon_{N-d}$. Endow $\K[[t]]^N$ with the module ordering defined in section \ref{sectionscission}. Then it is clear that the first $N-d$ columns form a standard basis. Its maximal order is given by $\varepsilon_{N-d}$. The greatest common divisor of all $(N-d)$-minors of $\p f(\bar a)$, which is invariant under equivalence of matrices, is equal to $\prod_i E_i t^{\varepsilon_i}$. Therefore $\varepsilon_{N-d}\leq e$. Constructing a scission $\sigma$ for $\ell(z)=\p f(\bar a)z$ as in section \ref{sectionscission} thus shows that $\sigma h$ is contractive. By Proposition \ref{pointwiserank} we see that the Rank Condition is fulfilled. This gives pointwise linearization of $\hat a \mapsto f(\bar a + \hat a)$ where $f=(f_1,\ldots, f_N)$.\\

We leave it to the reader to verify that one can carry out the same procedure as in section \ref{sectionhypersurfacecase} to prove the local triviality of $\pi_e: X_\infty \ra \pi_e(X_\infty)\subseteq (\K^n)^{e+1}$.



\subsection{On a local factorization theorem of Grinberg-Kazhdan and Drinfeld}\label{sectiondrinfeld}

Let $X$ be a scheme of finite type over $\K$ and denote by
$X_\infty[\gamma_0]$ the formal neighbourhood of the arc space
$X_\infty$ in a $\K$-arc $\gamma_0$ which does not lie in $(\Sing
X)_\infty$. If $Z$ is a scheme with structure sheaf $\Oo_Z$ and $I$ is
an ideal sheaf of a point $p\in Z$, then the formal neighbourhood of
$Z$ in $p$, denoted by $Z[p]$, is the topological space $\{p\}$
equipped with the structure sheaf $\varprojlim \Oo_Z/I^n$ (given
by $\widehat \Oo_{Z,p}$); see for example \cite{hartshorne},
Chap. II.9. For our purpose we will use the characterization of a
formal neighbourhood by its functor of points (see below). The
following theorem is proved for $\K=\C$ by Grinberg-Kazhdan in
\cite{Grinberg-Kazhdan} and for arbitrary fields $\K$ in
\cite{drinfeld} (for the notion of formal scheme we refer to \cite{hartshorne}, Chap. II.9.): 

\begin{Thm}\label{thm:GKD} 
Under the above assumptions we have
$$X_\infty[\gamma_0] \cong Y[y]\times D^\infty,$$

where $Y[y]$ is the formal neighbourhood of a (possibly singular)
scheme $Y$ of finite type over $\K$ in a suitable point $y\in Y$ and
$D^\infty$ is a countable product of formal discs, i.e., the formal
spectrum of $\K[[t]]$.
\end{Thm}

\begin{Rem}
Denote the structure sheaf of $X_\infty$ by $\Oo_\infty$. The assertion can be reformulated (see \cite{reguera}) as

$$\widehat \Oo_{\infty, \gamma_0} \cong B[z_i ; i \in \N]\widehat \ .$$

Here $B$ denotes a complete local Noetherian ring (corresponding to $Y[y]$) and $B[z_i ; i \in \N]\widehat \  $ is the completion of $B[z_i; i\in \N]$ w.r.t. the ideal $(z_i ; i \in \N)$.\\
\end{Rem}

We will show that in the case of hypersurfaces this local factorization theorem can be seen as a consequence of the Rank Theorem (over a field of characteristic $0$). The restriction to hypersurfaces is not essential but keeps notation simple. The relation between these two results is as follows: As in \cite{drinfeld} we use the characterization of the formal neighbourhood by its functor of points from test-rings to sets. A {\it test-ring} $A$ is a local, commutative, and unital $\K$-algebra with nilpotent maximal ideal $\n$ and residue field $\K$. The $A$-points of $X_\infty[\gamma_0]$ are $A[[t]]$-points of $X$ whose reduction modulo $\n$ coincides with $\gamma_0$.\\

Consider a hypersurface $X$ given by $f \in \K[x_1,\ldots, x_n,y]$ where $\K$ is a field of characteristic $0$. Let $\gamma_0=(x_0,y_0)\in \K[[t]]^{n+1}$ be a $\K$-arc of $X$, i.e., $f(\gamma_0)=0$. Moreover, we assume that $\gamma_0 \not\in (\Sing X)_\infty$, thus, w.l.o.g. $\w_t \p_y f(x_0,y_0)=d>0$ (the case $d=0$ is trivial and will henceforth be excluded). In this setting finding $A$-points for $X_\infty[\gamma_0]$ translates into finding $\gamma=(x,y)\in A[[t]]^n\times A[[t]]$ with $\gamma = \gamma_0 \mod \n$ and 

$$f(x,y)=0.$$

By the Weierstrass Preparation Theorem there are a unit $u \in A[[t]]^\times$ and a distinguished polynomial $q \in A[t]$, both unique, such that
$$\p_y f(x,y)=u\cdot q.$$ 
Taking this modulo $\n$ shows that $\deg_t q =d$ ($d$ depends on $\gamma_0$ but not on $\gamma$). For any integer $r > 1$ we define the textile isomorphism

$$\psi_q\colon A[[t]]^{n+1}\ra A[t]^n_{<(r+1)d}\times A[t]_{<rd}\times A[[t]]^n\times A[[t]]; (x,y)\mapsto (\bar x, \bar y, \xi, \eta)$$
induced by Weierstrass division $x=\bar x + q^{r+1}\xi$, $y=\bar y+q^r\eta$. Now work with $f\circ \psi_q^{-1}$. This will have the advantage that after fixing $\bar x$ and $\bar y$ the induced map is a quasi-submersion. Note: $\psi_q$ is not textile for general $A$ -- we are strongly using the nilpotency of $\n$ in $A$. By Taylor expansion of $\p_y f(\bar x+ q^{r+1}\xi, \bar y+q^r\eta)$ it is easy to see that 
\begin{equation}\label{eq:partial_is_multiple_of_q}
\p_y f(\bar x, \bar y)=\tilde u\cdot q
\end{equation} 
for some $\tilde u\in A[[t]]^\times$. Fixing $\bar x$ and $\bar y$ consider
$$\varphi\colon A[[t]]^n\times A[[t]] \ra A[[t]]; (\xi,\eta)\mapsto f(\bar x + q^{r+1}\xi, \bar y + q^r\eta)-f(\bar x, \bar y),$$ which has Taylor expansion

$$\varphi(\xi,\eta)=\p_x f(\bar x, \bar y)q^{r+1}\cdot \xi + \p_y f(\bar x, \bar y)q^r\cdot \eta + H(q^{r+1}\xi, q^r\eta),$$

where $H$ is at least of order two in its entries. Denote the linear part of $\varphi$ by $\ell$. From equation (\ref{eq:partial_is_multiple_of_q}) we deduce that $\im(\ell)=(q)^{r+1}$ and by construction for $h(\xi,\eta)=H(q^{r+1}\xi,q^r\eta)$ the relation $\im(h)\subseteq \im(\ell)$ holds, i.e., $\varphi$ is quasi-submersive. Therefore we don't have to check the rank condition. 
Set $e=\w_t (q)$ and define $\sigma$ to be the scission for $\ell$ constructed via division by $\tilde u q^{r+1}$. It is obvious that the contraction degrees satisfy $\kappa(\ell)=-(r+1)e$ and $\kappa(h)\geq 2re$. In the case that $e> 0$ the order condition is fulfilled: $\kappa(h)=2re>(r+1)e=\kappa(\sigma)$. If $e=0$ the situation is more complicated. The composition $\sigma h$ is still contractive with respect to the refined order $\W$ by Lemma \ref{lem:contractive_map_nilpotent_sufficient_criterion}. Therefore, the linearizing automorphism $u$ in Thm. \ref{rankthm} can be constructed. This allows linearization of $\varphi$. Solvability of the linearized equation
\begin{equation}\label{eq:linearized_equation}
\ell(\xi,\eta)=f(\bar x,\bar y) 
\end{equation}

is given, since by assumption $f(\bar x, \bar y)\in (q)^{r+1}$. In fact, the solution space to (\ref{eq:linearized_equation}) is a free $A[[t]]$-module of rank $n$.\\

Conversely, given $(q,\bar x, \bar y)\in A[t]\times A[t]_{<d(r+1)}\times A[t]_{<rd}$ fulfilling conditions (C) and equations (E) below, the induced map $\varphi$ is linearizable and the corresponding linear equation (\ref{eq:linearized_equation}) has a solution. The conditions and equations are as follows:

$$
\textnormal{(C)}\ \ \ \ \left\{
\begin{array}{ll}
 \textnormal{C}_q: & q=t^d \mod \n, q\ \textnormal	{monic}\\
 \textnormal{C}_{\bar x}: & \bar x \mod \n = x_0\mod(t)^{d(r+1)}\\
 \textnormal{C}_{\bar y}: & \bar y \mod \n = y_0\mod(t)^{dr}
\end{array}
\right.
$$

and 

$$
\textnormal{(E)}\ \ \ \ \left\{
\begin{array}{ll}
 \textnormal{E}_1: & \p_y f(\bar x,\bar y)= 0 \mod (q)\\
 \textnormal{E}_2: & f(\bar x,\bar y)=0 \mod (q^{r+1}).\\
\end{array}
\right.
$$

Indeed, by ($\textnormal{E}_1$) we see that $\p_y f(\bar x, \bar y)=B\cdot q$ for some $B\in A[[t]]$ (here we naturally identify $\bar x, \bar y$ with elements in $A[t]$). From (C) it follows that $B\in A[[t]]^\times$.  Therefore, $\im(\ell)=(q)^{r+1}$; $\varphi$ fulfills rank and order condition as follows analogously to the presentation above, and the equation 
\begin{equation}\label{eq:translated_drinfeld}
\varphi(\xi,\eta)=f(\bar x,\bar y)
\end{equation}

has a solution by ($\textnormal{E}_2$). In fact the set of solutions to equation (\ref{eq:translated_drinfeld}) is isomorphic to some $\Aarc_A^n$.\\

To sum up, any $A$-deformation of the arc $\gamma_0$ is determined by $(q,\bar x, \bar y)$ fulfilling conditions (C) and (E) and by $(\xi,\eta)$, which have to fulfill a linear equation. The first data defines a scheme $Y$ of finite type over $\K$. As in \cite{drinfeld} for any $\K$-algebra $R$ the set $Y(R)$ consists of triples $(q,\bar x, \bar y)$ fulfilling (C) and (E) (with $A$ replaced by $R$); the distinguished element $y\in Y$ appearing in the Grinberg-Kazhdan-Drinfeld formal arc theorem corresponds then to $(t^d,x_0 \mod (t)^{(r+1)d},y_0\mod (t)^{rd})$.



\subsection{On the solutions of polynomial differential equations} \label{polynomialvectorfields}
Textile maps appear naturally in the context of differential and difference equations. For $q,n\in \N$ and polynomials $P_i$ in $qn$ indeterminates we consider the following system of ordinary differential equations (with constant coefficients):

$$ (*) \left\{
\begin{array}{ccc}
x_1^{(q)} & = & P_1(x_1, \ldots, x_n, x_1^{(1)}, \ldots, x_n^{(1)},\ldots, x_1^{(q-1)},\ldots, x_n^{(q-1)})\\
\vdots & { } & {} \\
x_n^{(q)} & = & P_n(x_1, \ldots, x_n, x_1^{(1)}, \ldots, x_n^{(1)},\ldots, x_1^{(q-1)},\ldots, x_n^{(q-1)})
\end{array}
\right.
$$

We are searching for solutions $x=(x_1,\ldots, x_n) \in \K[[t]]^n$ where $x_i^{(l)}$ denotes the $l$-th derivative  (with respect to $t$) of a power series $x_i \in \K[[t]]$. The problem of finding solutions to this system can be formulated via cords and textile maps as follows: Identify again $\K[[t]]$ with $\Aarc$ and denote by $D$ the {\it differential operator} $$D:\Aarc \ra \Aarc; a=(a_i)_{i\in \N} \mapsto \left( (i+1)a_{i+1}\right)_{i\in \N}.$$
For $j\in \N$ its $j$-th power is written as $D^j$, $D^0:=\id$. We denote by $p_i$ the tactile map induced by $P_i$. Solutions to the system $(*)$ are zeros of the textile map $f:\Aarc^n \ra \Aarc^n$;

$$f(x_1,\ldots, x_n) = \left(
\begin{array}{ccc}
D^q(x_1) & - & p_1(x_1, \ldots, x_n,\ldots, D^{q-1}x_1,\ldots, D^{q-1}x_n)\\
$ $ & \vdots & {} \\
D^q(x_n) & - & p_n(x_1, \ldots, x_n,\ldots, D^{q-1}x_1,\ldots, D^{q-1}x_n)
\end{array}
\right)$$

Taking into account initial conditions $D^jx_i(0)=x_{0,i}^j$, $i=1,\ldots,n$, $j=0,\ldots, q-1$ yields a regular family $$f_\B: \B\times (\m^q\cdot \Aarc^n) \ra \Aarc^n,$$ where $\B=(\K^q)^{n}$. Denote by $f_{x_0}$ the restricted map $f_\B(x_0, -)$. Zeros of $f_{x_0}$ are solutions of the differential equation with initial condition $x_0 \in \B$.

\begin{Thm}\label{Thm:differential_equation}
Solutions of the polynomial differential equation $(*)$ with initial vector $x_0$ are isomorphic to solutions $y \in \m^q\cdot \Aarc^n $ of a linear system of the form $$\ell(y) = b_0,$$ where $\ell:\Aarc^n \ra \Aarc^n$ is textile linear  and $b_0 \in \Aarc^n$.
\end{Thm}

The proof is mostly a consequence of the following lemma:

\begin{Lem}\label{lemmavectorfield}
(a) Let $\nu_1,\ldots, \nu_p \in \N \cup \{0 \}$ and $\mu_1,\ldots, \mu_p \in \N$; the map $$g: \Aarc \ra \Aarc; x\mapsto \left(D^{\nu_1}x\right)^{\mu_1}\cdots \left(D^{\nu_p}x\right)^{\mu_p}$$ has contraction degree $\kappa(g)\geq-\max\{\nu_1,\ldots, \nu_p\}$.\\
(b) Let $U$ be a neighbourhood of $0$ and let $h:U \subseteq \Aarc \ra \Aarc$ be a textile map. Then $$\kappa(T_0 h) \geq \kappa(h).$$
\end{Lem}

\begin{proof} (Lemma)
(a) is obvious. For (b): write $h_i$ for the $i$-th coefficient of $h$. Then

\begin{eqnarray*}
T_0 h \cdot y = \p h \bullet y & = & \sum_{\gamma\in\N}(\p_\gamma h)(0)\cdot y_\gamma\\
& = & \left(\sum_{\gamma\leq i + \kappa(h)} (\p_\gamma h_i)(0) y_\gamma \right)_{i\in \N}.
\end{eqnarray*}
Hence, $\kappa(T_0h)\geq \kappa(h)$.

\end{proof}

To prove the theorem we proceed as follows: Using Taylor expansion gives $f_{x_0}$ as $f_{x_0}(y)=f(x_0) + \widehat f(x_0;y)$ with $\widehat f(x_0;0)=0$. This $\widehat f(x_0, -):\m^q\cdot \Aarc^n \ra \Aarc^n$ fulfills the assumptions of the Rank Theorem: According to (a) in Lemma \ref{lemmavectorfield}, $\kappa(p)\geq -(q-1)$, and further by (b) $\kappa(T_0 p)\geq -(q-1)$ follows. Since $T_0 \widehat f = D^q - T_0 p$, we deduce $\kappa(T_0 \widehat f) = - q$. By Construction \ref{existenceofscissions} there is a scission $\sigma$ for $T_0 \widehat f$ with $\kappa(\sigma)=+q$. Hence, $\sigma$ composed with $h:= f-T_0 \widehat f$ is contractive. The Rank Condition is fulfilled since $T_0 \widehat f$ is injective and, thus, flat (see Proposition \ref{flatisrank}).\\ By the Rank Theorem there exist linearizing automorphisms $u,v$ such that $v\widehat f u^{-1} =\ell$ with $\ell:= T_0 \widehat f$. Similar to section \ref{sectionhypersurfacecase} we conclude that solutions of $\ell(x_0; \eta)= v(-f(x_0))$ are isomorphic to solutions of $f(x_0 + u^{-1}(\eta))=0$.

\subsection{On Tougeron's Implicit Function Theorem} \label{sectiontougeron}
Using Theorem \ref{rankthm} we prove the following result, known as Tougeron's Implicit Function Theorem (see Proposition 1, p. 206 ff. in \cite{tougeron}):

\begin{Thm}
Consider variables $x=(x_1,\ldots, x_n)$ and $y=(y_1,\ldots,y_p)$; let $F \in \K[[x,y]]$ with $F(0,0)=0$ and $\K$ a complete valued field of characteristic $0$. Denote by $\delta$ the Jacobian of $F$ w.r.t. $y$ evaluated at $(x,0)$, by $A$ the Jacobian ideal of $F$, i.e. $A=(\delta_1,\ldots, \delta_p)$ with $\delta_i=\p_{y_i}F(x,0)$, and by $B$ any proper ideal in $\K[[x]]$. If $F(x,0)\in B A^2$, then there exist $y_1(x),\ldots, y_p(x) \in BA$ such that $F(x,y_1(x),\ldots, y_p(x))=0$. 
\end{Thm}

The proof of the theorem reduces to a proof of the following lemma (see \cite{tougeron}):

\begin{Lem}\label{tougeronlemma}
With the notation of the theorem and new indeterminates $y^i=(y^i_1,\ldots, y^i_p)$, $i=1,\ldots,r$, the following assertion holds: For any $r\leq p$ there are power series $Y^i \in \K[[x,y^l_j; 1\leq l \leq r, 1\leq j\leq p]]^p$, $1\leq i \leq r$, so that

$$F(x, \sum_{i=1}^r\delta_i Y^i)=F(x,0)+\delta(\sum_{i=1}^r\delta_i y^i).$$
\end{Lem}

The lemma obviously asserts that the tactile map $$g:(y^1,\ldots, y^r)\mapsto F(x,\sum_{i=1}^r \delta_iy^i)$$ can be linearized. As we know from the introduction, linearizing tactile maps in several variables is not always possible. In the situation of the lemma, however, we work with a quasi-submersion. Indeed, the textile map $f$ at hand is induced by a rather arbitrary $F \in \K[[x,y]]$, in which for $y$ a linear combination of $\delta_i$ is substituted. Thus, the image of $f$ is contained in the image of the tangent map given by $\delta$: $f$ is a quasi-submersion, the Rank Condition is obsolete and linearization is possible as soon as the order condition is fulfilled. Note: the Rank Theorem ensures linearization by textile maps. Thus we will have to check that the linearizing automorphisms obtained from the Rank Theorem  in this case are indeed tactile!\\

{\it Proof of Lemma \ref{tougeronlemma}:}
Write $F(x,y)=F(x,0)+\delta y + H(x,y)$ where $H$ is at least quadratic in $y$; set $z=(z^1_1,\ldots, z^r_1,\ldots, z^1_p,\ldots,z^r_p)$. Without loss of generality $F(x,0)=0$; otherwise consider $F(x,y)-F(x,0)$. Define the matrix $D$ as the $p\times p$ block matrix $\textnormal{diag}(\delta\_,\ldots, \delta\_)$, where $\delta\_=(\delta_1,\ldots, \delta_r)$. Then consider the tactile map

$$g:\K[[x]]^{rp} \ra \K[[x]]; z \mapsto F(x,D\cdot z).$$

The tangent map of $g$ with respect to $z$ at $0$ is given by $ $\p g(0)$=\p F(x,0)\cdot D$ and thus $I=\im\left(\p g(0)\right)=(\delta_i\delta_j; 1\leq i\leq p, 1\leq j\leq r)$. We first check the Rank Condition. The non-linear part of $g$ is $H(x,D\cdot z)$. Its image is obviously contained in $I$. So $\im (g) \subseteq I$. The automorphism $v$ of Theorem \ref{rankthm} is in this case defined as $v= \id - (\id - \p g(0) \sigma)g\sigma \p g(0) \sigma$ for a scission $\sigma$ of $\p g(0)$. But since $\id - \p g(0) \sigma$ is a projection onto the complement of $I$ we see that $v = \id$ on $\im(g)$, i.e., $g$ has constant rank. In particular $v$ is a tactile map.\\
It remains to prove that $\sigma h$ is contractive. Here we use a scission $\sigma$ as constructed in section \ref{sectionscission}. For this write $h(x,z)=\sum_{1\leq i,j\leq r} \delta_i \delta_j h_{ij}(x,z)$ for appropriate $h_{ij} \in \K[[x,z]]$. Note that in the definition of $h_{ij}$ we keep track of the order of $i,j$; this is just a technical convention. The scission $\sigma$ is given by division through the $\delta_i\delta_j$, $1\leq i,j \leq r$. Since $h(x,z)$ has this special form there is a natural way how to divide: $\sigma h(x,z)$ has on position $(i-1)\cdot r + j$ the term $h_{ij}$ and $0$ else. It's easy to see that $\kappa(\sigma h)>0$, since the $h_{ij}(x,z)$ are at least of order $2$ in $z \in (x_1,\ldots,x_n)\subseteq \K[[x]]$. Moreover, $\id + \sigma h$ is tactile, since the $h_{i,j}(x,z)$ are power series in $x, z$. Thus the lemma follows from Theorem \ref{rankthm}. \qed

\medskip

\subsection{On Wavrik's Approximation Theorem} \label{sectionwavrik}

Let $F\in \K[[x,y]][z]$ with indeterminates $x=(x_1,\ldots,x_n)$, $y=(y_1,\ldots, y_p)$ and $z=(z_1,\ldots, z_r)$. We are now interested in solutions $(y(x),z(x)) \in \K[[x]]^{p+r}$ to $F(x,y,z)=0$. The solution is called an {\it $N$-order solution} if $\w F(x,y(x),z(x)) \geq N + 1$. The following assertion has been proved for $F$ polynomial in \cite{artin2} and for $F$ a power series in \cite{wavrik_general}: Given a positive integer $q$, there is an $N_0 = N_0(F,q)$ such that any $N$-order solution,  $N\geq N_0$, can be approximated to order $q$ by an exact solution. Let us review this theorem in the special case $n=p=m=1$, $F$ an irreducible polynomial as is proven in Theorem I of \cite{wavrik_plane}:

\begin{Thm}\label{thm:artin_wavrik}
Given $q >0$ and $F(x,y)\in \K[x,y]$ irreducible there exists an $N$ such that if $\bar y\in \K[[x]]$ satisfies $F(x,\bar y)\equiv 0 \mod (x)^N$, then there exists a series $y(x)\in \K[[x]]$ with $F(x,y(x))=0$ and $y(x)\equiv \bar y \mod (x)^q$.
\end{Thm}

\begin{proof}
Using the following argument (see \cite{wavrik_plane}, proof of Theorem I) we can further assume that $\w_x \p_yF(x,\bar y)< r \in \N$: Consider the discriminant $\Delta \in \K[x]$ of $F$ and $\p_y F$ and polynomials $A(x,y)$ and $B(x,y)$ such that $\Delta = AF + B\p_y F$. Denote by $d$ the order of $\Delta$. If $\w F(x,\bar y) > d$ then $\w \p_y F(x,\bar y) < d$.\\
Take an element $y_0\in \K[[x]]$, $\deg y_0 \leq d$, such that $\w F(x,y_0)\geq d+1$ and $\w \p_y F(x,y_0)=e < d$. The textile map $g:(x)^{d+1}\subseteq \K[[x]] \ra \K[[x]]$, $g(z)=F(x,y_0 + z)$ has constant rank, since $\im(F)\subseteq \im(\p_y F(x,y_0))=(x)^e$. It is easy to see that $\kappa(\sigma h)>0$, where $\sigma$ is a scission of $\p_y F(x,y_0)$ as in section \ref{sectionscission} and $h(z)= g(z)-\p_y g(y_0)z$. Thus we may linearize $g$ at $0$ and lift any such $y_0$ to a solution $y(x)$ of $f(x,y)=0$. This linearization is possible as soon as $y_0$ is known up to order $d$.
\end{proof}

\begin{Rem}
One should note that the index $N$ obtained in the proof by using the Rank Theorem is much lower than the one provided in \cite{wavrik_plane}, p. 411. The proof provided there needed $N=2d+q$. This follows from the fact that Tougeron's Implicit Function Theorem is used, which is not as efficient as the Rank Theorem. In a subsequent section of \cite{wavrik_plane} a minimal $N$ is computed by means of resolution of singularities. It is not clear whether this optimal $N$ is the same as the one obtained in the proof above. 
\end{Rem}

\medskip

\subsection{On an inversion theorem by Lamel and Mir} \label{sectionmirlamel}

\def\O{{{\cal O}}}
\def\Aut{{\rm Aut}}

In their paper \cite{lamel_mir}, Lamel and Mir investigate the group of local holomorphic automorphisms of $\C^n$ preserving a real analytic submanifold $M$ (Cauchy-Riemann automorphisms). Their key result to describe this group is as follows. Denote by $\Oo_n$ the space of holomorphic germs $(\C^n,0) \ra \C$, and by $\Aut(\C^n,0)\subset\Oo_n^n$ the group of biholomorphic germs of $(\C^n,0)$. A map $f \in \Oo_n^n$ is said to have {\it generic rank $n$} if its Jacobian determinant $\det (\p f)$ is nonzero as an element in $\Oo_n$. Any $f\in \Oo_n^n$ induces a  map $F:\Aut(\C^n,0) \ra \Oo_n^n$ given by the composition $u \mapsto f\circ u$.

\begin{Thm}\label{mirlamelthm}
(Thm. 2.4 in \cite{lamel_mir}) Let $f \in \Oo_n^n$ be a germ of a holomorphic map of generic rank $n$. There exists a holomorphic map $\Psi: \Oo_n^n \times \hbox{\rm GL}_n(\C) \ra \Aut(\C^n,0)$ inverting $F$ in the sense that $\Psi(F(u), \p u(0))=u$ for all $u\in \Aut(\C^n,0)$. Furthermore $\Psi$ can be chosen such that $\p (\Psi(b,\lambda))(0)=\lambda$ for all $b\in \Oo_n^n$ and $\lambda \in \hbox{{\rm GL}}_n(\C)$.
\end{Thm}

\begin{proof} We shall show how the Rank Theorem allows to construct a suitable $\Psi$ in the formal context taking the space $\Ce_n=\C[[x_1,\ldots,x_n]]$ of formal power series (cords) instead of $\Oo_n$.  In order to treat convergent power series one has to  refer to the Rank Theorem from \cite{HM}, with the same reasoning as in the formal case. The argument in \cite{lamel_mir} is much more involved, their computations invoking tacitly various instances of the proof of the Rank Theorem.

The map $F:\Ce_n^n \ra \Ce_n^n$ is tactile with Jacobian determinant $\det(\p f) \neq 0 \in \Ce_n$. Our goal is to linearize it locally at the linear part $\lambda x$ of a given $u$.  Decompose  $u\in \Aut(\C^n,0)$ as $u(x) = \lambda \cdot x + v(x)$ with $\lambda\in {\rm GL}_n(\C)$ and $v\in\Ce_n^n$ of order $2$. Consider the Taylor expansion (with the obvious abbreviation) 
$$G(x)=f(u(x)) - f(\lambda x) = \partial f(\lambda x)\cdot  v + \partial^2 f(\lambda x)\cdot v^2 + \ldots ,$$

with linear part $\ell(v)=\partial f(\lambda x)\cdot  v$. Since $f$ is of generic rank $n$ and $\lambda \in \hbox{\rm GL}_n(\C)$, the map $\ell$ is injective. Thus $G$ has constant rank at $0$, by the remark after Proposition \ref{flatisrank}. Next we show that $G$ fulfills the order condition of the Rank Theorem (Thm. \ref{rankthm}). In order to keep notation short we assume that $\lambda$ is the identity matrix. Denote the initial form of lowest degree of $f^i$ by $f^i_*$ and write $e_i$ for the order of $f^i$. By the same triangularization argument as in Lemma 4.5 of \cite{lamel_mir} we may assume that $f_*=(f^1_*, \ldots, f^n_*)$ has generic rank $n$. This gives
 
$$\w \det(\p f_*) = e_1 + \ldots + e_n - n.$$ 

Moreover for $\beta \in \N^n$, $|\beta|=m$, and  $\ l \in \{1,\ldots,n\}$ we have
$$\w (\p_\beta f^l_*) \geq e_l-m$$ 

for $e_l \geq m$ and $\w (\p_\beta f^l_*)=\infty$ else. As the map $f$ has generic rank $n$ a scission $\sigma$ for $\ell$ can be obtained as in Proposition \ref{genericrankscission} by the adjoint matrix of $\partial f$. The nonlinear part $h=G-\ell$ consists of all terms of $G$ which are of order at least $2$ in $v_1, \ldots, v_n$. We write 
$$h(v)=\sum_{{\alpha \in \N^n}\atop {|\alpha|\geq 2}} h_\alpha\cdot v^\alpha,$$

with $h_\alpha \in \Ce_n^n$. For $|\alpha|=m\geq 2$, the coefficient vector $h_\alpha$  has $i$-th component $h_\alpha^i$ involving the $m$-th order partial derivatives of $f^i$. Multiplication with the adjoint matrix of $\p f$ yields in the $i$-th component

\begin{eqnarray*}
\left((\p f)^{adj} \cdot h_\alpha\right)_i  & = & \sum_{l=1}^n (\p f)_{il}^{adj}\cdot h_\alpha^l\\
& = & \sum_{j=1}^n (-1)^{i+l}\det (\p f)^{(l,i)}\cdot h_\alpha^l.
\end{eqnarray*}

Here, $A^{(i,j)}$ denotes the $(n-1)\times (n-1)$ matrix obtained from an $n\times n$ matrix $A$ by deleting the $i$-th row and the $j$-th column. From the last equation it follows that each term in the sum is of type

$$T=(\prod_{1\leq j \leq n\atop j\neq l} \p_{i_j} f^j) \cdot (\p_\beta f^l),$$

where $i_j\in \{1,\ldots,n\}$ and $\beta \in \N^n$, $|\beta|=m$. Obviously 
$$\w T \geq \sum_{i=1}^n e_i - n - m + 1 .$$

This allows to compute the contraction degree $\kappa$ of the map 
$$\varphi_\alpha: v \mapsto \frac{1}{\det (\p f)}\cdot (\p f)^{adj}\cdot h_\alpha \cdot v^\alpha$$ 
if $v$ varies in $(x_1,\ldots, x_n)^2\cdot \Ce_n^n$. We get

\begin{eqnarray*}
\kappa(\varphi_\alpha) 
& \geq & - (\sum_{i=1}^n e_i - n) + (\sum_{i=1}^n e_i - n -m + 1) + (2m -2)\\
& = & m-1.
\end{eqnarray*}

So $\kappa(\varphi_\alpha)>0$ for all $\alpha$ with $|\alpha| =m\geq 2$. Now

$$(\sigma h)(z)=\sum_ {\alpha \in \N^n, |\alpha|\geq 2} \varphi_\alpha(z)$$ 

yields

$$\kappa(\sigma h)\geq \min_\alpha \kappa(\varphi_\alpha) = 1 > 0.$$

The order condition is fulfilled. Therefore  $G$ can be linearized to $\ell$ by local automorphisms of source and target. Once this is done, it suffices to solve the linear equations in order to construct the required map $\Psi$ of the theorem.
\end{proof}

\section{Outlook and unsolved problems}

We conclude the article by a collection of questions related to
the linearization principle and its applications.\\

(a) {\it The Grinberg-Kazhdan-Drinfeld formal arc theorem:} Theorem
\ref{thm:GKD} describes a factorization of the formal neighbourhood of
an arc (not lying completely in the singular locus of the base
variety) into an infinite dimensional smooth part and a formal
neighbourhood $Y[y]$ of a possibly singular scheme of finite
type $Y$ in a point $y\in Y$. Unfortunately, the proofs presented in
\cite{Grinberg-Kazhdan}, \cite{drinfeld} and
section \ref{sectiondrinfeld} rely on a non-constructive functorial
description of the formal neighbourhood. In simple examples it is
possible to compute the decomposition explicitly:

\begin{Ex} (cf. \cite{drinfeld}) Let $X$ be the
hypersurface $$g=yx_{n+1} + f(x_1,\ldots, x_n)=0,$$ $f(0)=0$, and take
$\gamma_0=(0,\ldots, 0, t,0)\in X_\infty$ where the first $n+1$
components correspond to the variables $x_i$, $1\leq i\leq n+1$. Note
that $\p_y g(\gamma_0)=t$, thus $\gamma_0\not\in
(X_{\sing})_\infty$. In this example we can identify $Y$ with $\Spec
\K[x_1,\ldots, x_n]/(f)$ and $y$ with $0$. To see this, let $A$ be a
test-ring with maximal ideal $\m$. We are interested in the
$A$-deformations of $\gamma_0$, i.e., in $$\gamma=(x_1(t),\ldots,
x_{n+1}(t),y(t))\in A[[t]]^{n+2}$$ with $\gamma \mod \m =\gamma_0$. By
the Weierstrass Preparation Theorem any $A$-deformation $x_{n+1}(t)$
of $t$ can be written as $$x_{n+1}(t)=u(t-\alpha),$$ with $\alpha \in
\m$ and $u\in A[[t]]^\times$. It's easy to show the following: Given
$\alpha, u$ and $x_1(t),\ldots, x_n(t) \in \m[[t]]$ there is at most
one $y(t) \in \m[[t]]$ such that $$y(t)x_{n+1}(t) + f(x_1(t),\ldots,
x_n(t))=0.$$ Moreover, $y(t)$ exists if and only if
$f(x_1(\alpha),\ldots, x_n(\alpha))=0$ holds. Thus, any
$A$-deformation of $\gamma_0$ is determined by $\alpha \in \m$,
$(x_1(t),\ldots, x_n(t))\in \m[[t]]^n$ with the
property $$f(x_1(\alpha),\ldots, x_n(\alpha))=0$$ and $u\in
A[[t]]^\times$ (note: $(x_1(\alpha),\ldots, x_n(\alpha))$ is nothing
but an $A$-point of $\Spec \K[x]/(f)$ equal to $0$ modulo
$\m$). Conversely, any $A$-deformation $\beta$ of $0$ in $Y$ gives
rise to an $A$-deformation of $\gamma_0$ in the following way: choose
$(x_1(t),\ldots, x_n(t))\in \m[[t]]^n$ such that $$f(x_1(t),\ldots,
x_n(t))\in (t-\beta)\cdot \m[[t]].$$ This involves free choice of
infinitely many coefficients of the $x_i(t)$ which contributes to the
$D^\infty$ part in Theorem \ref{thm:GKD}. Next we take an arbitrary
$u\in \A[[t]]^\times$, once more the coefficients of $u$ contribute to
$D^\infty$. Finally, we determine $y(t)$ uniquely
by $$y(t)=u^{-1}(t-\beta)^{-1}f(x_1(t),\ldots, x_n(t)).$$
\end{Ex}

It would be desirable to be able to compute explicitly the
decomposition of the formal neighbourhood for the general
case. Moreover,  we would like to identify the finite type part $Y[y]$. What is
its geometric significance (for the arc scheme and the base scheme)?\\

(b) {\it Power series solutions to differential equations:} Describe
the geometry of power series solutions to non-linear (ordinary and
partial) differential equations by means of the rank theorem. A first
step has been carried out in section \ref{polynomialvectorfields},
where we considered explicit ordinary differential equations with
constant coefficients. Examples of more complicated cases
are: $$tx'+x^2=0$$ or $$t(x')^2 + x^2 =0.$$ In addition one could try
to regain and extend the results in \cite{denef_lipshitz}. There Denef
and Lipshitz prove among other results a version of Artin's
approximation theorem (see section \ref{sectionwavrik}) for algebraic
differential equations: any formal power series solution to an
algebraic differential equation can be approximated (in the \m-adic)
topology by a differentially algebraic power series.\\
 
(c) Implicit/Inverse function theorem for textile maps: textile maps
are a natural generalization of maps given by substitution of power
series. Theorem \ref{inversemappingthm} shows invertibility of textile
maps which are (up to a linear coordinate change) of the form $\id +
h$, where $h$ is contractive (see section
\ref{firstdefinitions}). Several lines of further investigation lie at
hand: first, one could try to weaken the contractivity assumption,
probably adding some additional structure to the coefficient ring
$\K$. In the case of the formal arc theorem (section
\ref{sectiondrinfeld}) we could weaken the assumptions on $h$ to
having contraction degree $\kappa(h)\geq 0$ as soon as $\K$ was
replaced by a test-ring (which especially implied completeness
w.r.t. to the topology defined by the powers of its maximal
ideal). Second, if we assume that $\K$ is a complete valued field, we
might ask for properties of $h$ ensuring convergence of the image of a
convergent cord: we call a cord convergent if its coefficients
are the coefficients of a convergent power series. Under which
conditions is $h(c)$ convergent if $c$ is? In the special case
of tactile maps over $\C$ a positive answer can be found in \cite{HM},
Thm. 6.2. As a variation, one could rise the same question for
``algebraic cords'' $c$. A cord is called {\it algebraic} if its
coefficients match with the coefficients of an algebraic power
series. These last generalizations constitute an important step
towards an approximation theorem for textile maps (see below). In view
of this, an adapted version of Tougeron's implicit function theorem for
textile maps is another goal to strive for.\\

(d) Artin Approximation: the approximation theorem
(Thm. \ref{thm:artin_wavrik}) shows that any sufficiently good
approximate power series solution to a formally analytic system of equations can be
approximated (in the \m-adic topology) by an actual solution. Moreover, in
case of polynomial systems of equations, any power series solution can
be approximated by algebraic solutions (see \cite{artin2}). Both
results are local in nature, since they focus on the structure of the
solution set locally (in the \m-adic topology) at an approximate/formal solution. How does the
geometry of the set of power series solutions look like in general? The
Denef-Loser local triviality result (Theorem
\ref{localtrivialitytheorem}) can be seen as an answer for the case
of power series solutions in one variable to a polynomial system with
constant coefficients.\\ Furthermore, viewing power series as cords,
it is natural to ask for an approximation result for textile maps: Let
$c$ be an approximative solution to the textile equation $f=0$. Is it
possible to approximate $c$ by an actual solution? Or, under which conditions
can solutions of a textile map be approximated by convergent (or
algebraic) ones? In contrast to the case of tactile equations (as they
appear in Artin's result), the answer will be negative in general. So
the question is, under which conditions do such approximation results
hold?

\bigskip

\bibliographystyle{alpha}
\bibliography{literatur}

\end{document}